\documentclass[11pt,reqno]{amsart}
\usepackage{fullpage,times,graphicx,amssymb,amsmath,psfrag,xcolor,natbib}
\usepackage{amsthm}
\usepackage{thmtools}
\usepackage[colorlinks,citecolor=bbluegray,linkcolor=ddarkbrown,urlcolor=blue,breaklinks]{hyperref}
\usepackage[nameinlink,capitalize,noabbrev]{cleveref}
\usepackage{bbding}
\usepackage[Tol]{colorblind}
\usepackage{algorithm,algorithmicx,algpseudocode}

\algnewcommand{\IIf}[1]{\State\algorithmicif\ #1\ \algorithmicthen}
\algnewcommand{\EndIIf}{\unskip\ \algorithmicend\ \algorithmicif}

\usepackage{booktabs}
\usepackage{tabularx}
\usepackage[table]{xcolor}
\usepackage{array}

\definecolor{contribution}{RGB}{235,243,250}

\newif\ifshowtodos
\showtodostrue

\usepackage{subcaption}
\usepackage{pgfplots}
\pgfplotsset{compat=1.18} 

\usepackage{pgf}
\usepackage{lmodern}

\usepackage{soul}

\usepackage[off]{auto-pst-pdf} 
\usepackage{comment}

\definecolor{ddarkbrown}{rgb}{0.5,0.2,0.05} \definecolor{bbluegray}{rgb}{0.05,0,0.5}

\newtheorem{theorem}{Theorem}[section]
\newtheorem{proposition}[theorem]{Proposition}
\newtheorem{definition}[theorem]{Definition}
\newtheorem{lemma}[theorem]{Lemma}
\newtheorem{corollary}[theorem]{Corollary}
\renewenvironment{proof}{\textbf{Proof.}}{\QED\bigskip}

\newcommand{\norm}[1]{\left\|#1\right\|}
\newcommand{\F}{\mathrm{F}}

\DeclareMathOperator{\sgn}{sgn}

\newcommand{\BEAS}{\begin{eqnarray*}}
\newcommand{\EEAS}{\end{eqnarray*}}
\newcommand{\BEA}{\begin{eqnarray}}
\newcommand{\EEA}{\end{eqnarray}}
\newcommand{\BEQ}{\begin{equation}}
\newcommand{\EEQ}{\end{equation}}
\newcommand{\BIT}{\begin{itemize}}
\newcommand{\EIT}{\end{itemize}}
\newcommand{\BNUM}{\begin{enumerate}}
\newcommand{\ENUM}{\end{enumerate}}

\newcommand{\BA}{\begin{array}}
\newcommand{\EA}{\end{array}}

\newcommand{\reals}{{\mathbb R}}

\newcommand{\Expect}{\operatorname{\bf E}}  

\newcommand{\Prob}{\mathop{\bf Prob}}

\newcommand{\QED}{~~\rule[-1pt]{6pt}{6pt}}

\newcommand{\argmin}{\mathop{\rm argmin}}

\newcommand{\dom}{\mathop{\bf dom}}

\let \oldsection \section
\renewcommand{\section}{\vspace{3ex plus 1ex}\oldsection}

\usepackage[inkscapelatex=true]{svg}

\begin{document}
\title{Restart and Adaptive Acceleration\\ in Stochastic Gradient Methods}

\author{Ali Elhishi}
\address{CNRS \& D.I., UMR 8548,\vskip 0ex
\'Ecole Normale Sup\'erieure, Paris, France.}
\email{ali.abdelfatah-elhishi@inria.fr}

\author{Christophe Roux}
\address{Zuse Institute Berlin \& TU Berlin,\vskip 0ex
Berlin, Germany.}
\email{roux@zib.de}

\author{Alexandre d'Aspremont}
\address{CNRS \& D.I., UMR 8548,\vskip 0ex
\'Ecole Normale Sup\'erieure, Paris, France.}
\email{aspremon@ens.fr}

\date{\today}
\subjclass[2010]{}

\begin{abstract}
We study restart schemes in stochastic optimization problems for non-smooth and weakly convex functions that satisfy a Kurdyka-\L ojasiewicz (K{\L}) inequality. Using restarts allows us to leverage K{\L} inequalities to achieve optimal convergence rates, with acceleration depending explicitly on the K{\L} exponent. Furthermore, for stochastic gradient descent (SGD), optimal restart schedules correspond to piecewise constant (step decay) step sizes. While regularity constants such as the K{\L} exponent are typically unknown in practice, we prove that restart schemes are robust to significant misspecification of these constants, hence nearly adaptive. We detail numerical experiments on a number of problems where the K{\L} exponent is controlled.
\end{abstract}
\maketitle

\section{Introduction}\label{s:intro}
We focus on solving the following optimization problem, written
\BEQ\label{eq:prob}
\BA{ll}
\mbox{minimize} & f(x)\\
\mbox{subject to} & x \in \mathcal{X}
\EA\EEQ
in the variable $x\in\reals^d$, where $\mathcal{X}\subset \reals^d$ is a closed convex set. We assume that $f$ is weakly convex and that it satisfies a Kurdyka--{\L}ojasiewicz (K{\L}) type inequality \citep{Bolt07}. Importantly, we do not assume that the function is smooth or even differentiable. Our goal here is to apply restart schemes to the stochastic gradient method to exploit the K{\L} inequality and improve convergence rates, while showing that the resulting method is in fact implementable, in the sense that it is nearly adaptive to restart parameters. 

For non-smooth functions, weak convexity \citep{Vial83} is arguably among the minimal assumptions allowing convergence rates to be derived for the stochastic gradient method \citep{Davi18}. Assuming weak convexity means that the Moreau envelope of the function can be computed explicitly and defines a notion of near-stationarity. This is the fundamental approach of \citet{Davi18}, where the Moreau envelope and the proximal map are used to analyze convergence but do not appear in the algorithm itself. This yields convergence rates in a much broader setting than that induced by more classical smoothness assumptions. Note that $L$-smoothness implies $L$-weak convexity, but the converse does not hold in general.

K{\L} inequalities, on the other hand, remain comparatively underexploited in the analysis of stochastic gradient descent (SGD) methods, despite the fact that they hold very generically and enable significantly faster convergence rates. These inequalities date back to the work of \citet{Loja63} and \citet{Kurd98}. \citet{Nemi85} show optimal complexity bounds for smooth convex optimization using K{\L}-like assumptions and restart schemes. In the same setting, \citet{Roul17} show adaptivity of these schemes to the K{\L} parameters. The seminal work of \citet{Bolt07} generalizes these bounds to subanalytic functions and derives early convergence results, extended by \citet{Atto13,Bolt17a} to a much broader class of methods. In the deterministic case, using a mix of assumptions similar to ours, namely weak convexity and linear sharpness (which imply K{\L} inequalities with exponent 0), \citet{Davi20a} show linear convergence of a Polyak subgradient method on phase retrieval problems. \citet{Li18} derive calculus rules for the K{\L} exponent.

Several papers study restart schemes in the convex setting, with the optimal complexity bounds of \citet{Nemi85} cited above, the experiments of \citet{ODon15}, the bounds of \citet{Gise14,Kim18}, and the results of \citet{Yang18}, which use K{\L} inequalities to obtain improved rates. In the nonconvex setting, \citet{Losh16} empirically study restart schemes for SGD to train neural networks, while \citet{Zhou20} apply restart schemes to a proximal gradient algorithm in a composite setting. More recently, in the smooth nonconvex case and assuming that K{\L}-like conditions hold, \citet{Font21,Fatk22} show improved SGD complexity bounds using polynomially decaying step sizes, and \citet{Fatk22} obtain faster rates by combining variance reduction with restarts. \citet{Wang22} study restarts for accelerated stochastic gradient descent to minimize finite sums, while \citet{Li23} study restarted accelerated gradient methods, both in the smooth setting. Finally, linear convergence bounds can be obtained in certain sharp instances \citep[see, e.g.,][]{Davi18a,Davi19}.

Our contributions are as follows.

\begin{itemize}
\item \textbf{Optimal rates.} We show that restarting the proximal stochastic subgradient method on $\rho$-weakly convex functions that are K{\L} with exponent $\theta\in[1/2,1]$ yields $\Expect[f_\lambda(x_T)]-f^*=\mathcal{O}(T^{-1/(4\theta-1)})$ (\Cref{sec:restart}). This improves on the $\mathcal{O}(T^{-1/2})$ baseline of \citet{Davi18} and matches the rates of \citet{Font21,Fatk22} without their smoothness assumption. A matching lower bound shows that these rates are optimal (\Cref{sec:minimax-optimality}).

\item \textbf{Momentum and restart-free variants.} The same rates hold for the stochastic subgradient method with momentum of \citet{Mai20} (\Cref{sec:momentum}), and can also be obtained without restarts, using step sizes directly adapted to the K{\L} exponent (\Cref{sec:accel-without-rest}).

\item \textbf{Near adaptivity.} Since the K{\L} parameters are typically unknown, we study the effect of misspecifying the restart scheme (\Cref{sec:adapt}): a constant restart schedule loses only a $\log\log T$ factor, and we quantify how the bounds degrade when the number of restarts or the step size scale is misspecified, showing that the restart schemes are robust to both.

  \item \textbf{Numerical experiments.} We report results on problems where the K{\L} exponent is known and controlled, including deep matrix factorization, for which we establish weak convexity and a K{\L} inequality, as well as on neural network training, illustrating the trade-off in the choice of restart frequency predicted by our bounds (\Cref{sec:numerical}).
\end{itemize}

\begin{table}[ht]
    \centering
    \scriptsize
    \setlength{\tabcolsep}{3.5pt}
    \renewcommand{\arraystretch}{1.35}

    \caption{Comparison of convergence guarantees for stochastic optimization methods. Here, $T$ denotes the total number of stochastic oracle calls, $N$ the number of outer iterations, and $\lambda=(2\rho)^{-1}$. Entries marked --- indicate that the corresponding guarantee is not established by the cited work. The results of \citet{Font21,Fatk22} (\textcolor{red}{in red}) control the optimality gap of the original objective, whereas our results control the optimality gap of the Moreau envelope. The result of \citet{Mai20} is stated for the auxiliary iterate $\bar{x}_T$ defined in their paper.}
    \label{tab:convergence-rates}

    \begin{tabularx}{\linewidth}{
        >{\raggedright\arraybackslash}p{3.0cm}
        >{\raggedright\arraybackslash}X
        >{\centering\arraybackslash}p{2.8cm}
        >{\centering\arraybackslash}p{3.0cm}
        >{\centering\arraybackslash}p{3.0cm}
    }
        \toprule
        \textbf{Method}
        & \textbf{Main assumptions}
        & \textbf{Restart schedule}
        & $\Expect[f_\lambda(x_T)]-f^*$
        & $\Expect[ \|\nabla f_\lambda(x_T)\|^2 ]$
        \\
        \midrule

        \citet{Davi18}
        & $\rho$-weakly convex
        & No restart
        & ---
        & $\mathcal{O}(T^{-1/2})$
        \\

        \citet{Mai20}
        & $\rho$-weakly convex
        & No restart
        & ---
        & $\mathcal{O}(T^{-1/2})$
        \\

        \citet{Font21,Fatk22} (SGD)
        & Smooth nonconvex; K{\L}
        & No restart
        & \textcolor{red}{$\mathcal{O}\!\left(T^{-1/(4\theta-1)}\right)$}
        & ---
        \\

        \citet{Fatk22} (PAGER)
        & Smooth nonconvex; K{\L}; average smoothness of stochastic gradients
        & Restarted PAGE
        & \textcolor{red}{$\mathcal{O}\!\left(T^{-1/(2\theta)}\right)$}
        & ---
        \\

        \rowcolor{contribution}
        \textbf{Exponential restart (ours)}
        & $\rho$-weakly convex; K{\L} 
        & $t_k\asymp \exp\!\left((4\theta-1)\gamma k\right)$
        & $\mathcal{O}(T^{-1/(4\theta-1)})$
        & $\mathcal{O}( T^{-2\theta/(4\theta-1)} )$
        \\

        \rowcolor{contribution}
        \textbf{Exponential restart with momentum (ours)}
        & $\rho$-weakly convex; K{\L}
        & $t_k\asymp \exp\!\left((4\theta-1)\gamma k\right)$
        & $\mathcal{O}(T^{-1/(4\theta-1)})$
        & $\mathcal{O}( T^{-2\theta/(4\theta-1)} )$
        \\

        \rowcolor{contribution}
        \textbf{Constant restart (ours)}
        & $\rho$-weakly convex; K{\L} 
        & $t_k=t=T/N$
        & $\displaystyle \mathcal{O}\!\left[ \left(\frac{\log{\log{T}}} {T}\right)^{\frac{1}{4\theta-1}} \right]$
        & $\displaystyle \mathcal{O}\!\left[ \left(\frac{\log{\log{T}}}{T}\right)^{\frac{2\theta}{4\theta-1}} \right]$
        \\

        \bottomrule
    \end{tabularx}
\end{table}

The paper is organized as follows. \Cref{sec:kl} recalls K{\L} inequalities and discusses their use in a stochastic setting. \Cref{sec:restart} derives improved complexity bounds for SGD in the weak convexity / K{\L} setting using restart schemes and shows that the resulting rates are optimal. \Cref{sec:momentum} extends these results to the stochastic subgradient method with momentum, while \Cref{sec:accel-without-rest} derives (nonadaptive) complexity bounds without restarts, using step sizes adapted to the K{\L} exponent. \Cref{sec:adapt} discusses adaptation results in the presence of misspecified restart schemes. Finally, \Cref{sec:numerical} presents numerical results on problems where the K{\L} exponent is known.

\section{Kurdyka--{\L}ojasiewicz inequalities \& stochastic optimization}\label{sec:kl} 
We begin by recalling the K{\L} inequality and discussing its application in stochastic optimization. For a more complete treatment of this property, see, e.g., \citet{Bolt07}. We start with a generic definition taken from \citet{Atto13}.

\begin{definition}[Kurdyka--{\L}ojasiewicz]\label{def:kl}
A function $f : \reals^d \to \reals \cup \{+\infty\}$ is said to have the Kurdyka--{\L}ojasiewicz property at
\(x^\ast \in \operatorname{dom}\partial f\) if there exist \(\eta \in (0,+\infty]\), a neighborhood \(U\) of \(x^\ast\), and a continuous concave function $\varphi : [0,\eta) \to \mathbb{R}_+$, which we refer to as a \emph{desingularizing} function, such that, 

\begin{enumerate}
    \item \(\varphi(0)=0\),
    
    \item \(\varphi\) is \(C^1\) on \((0,\eta)\),
    
    \item for all \(s \in (0,\eta)\), \(\varphi'(s)>0\),
    
    \item for all
    \[
    x \in U \cap \left[ f(x^\ast) < f(x) < f(x^\ast)+\eta \right],
    \]
    the Kurdyka--{\L}ojasiewicz inequality holds:
    \BEQ\label{eq:KL-generic}
    \varphi'(f(x)-f(x^\ast))
    \operatorname{dist}\!\left(0,\partial f(x)\right)
    \ge 1.
    \EEQ
\end{enumerate}
\end{definition}

An important class of desingularizing functions is given by power functions $\varphi(s)= \tilde cs^{1-\theta}$ with $\theta \in [0,1)$, in which case~\eqref{eq:KL-generic} specializes to $\operatorname{dist}\!\left(0,\partial f(x)\right) \geq \frac{1}{\tilde c(1-\theta)} (f(x)-f(x^*))^\theta$, which we can simplify to
\BEQ\label{eq:KL}
\operatorname{dist}\!\left(0,\partial f(x)\right) \geq c \left(f(x)-f(x^*)\right)^\theta.
\EEQ
Although the standard power-type desingularizing functions correspond to exponents $\theta \in [0,1)$, it is sometimes useful to also consider the limiting case $\theta=1$ in the form of the K{\L}-type growth condition
\[
\operatorname{dist}\!\left(0,\partial f(x)\right) \geq c \left( f(x)-f(x^*)\right)
\]
for some constant $c>0$, locally around $x^*$ and for $f(x)>f(x^*)$. This case is not induced by a power desingularizing function satisfying the usual K{\L} definition, but it appears naturally for certain flat functions. In particular, there are functions for which such an inequality holds with exponent $\theta=1$, while no inequality of the above power form holds for any $\theta<1$. For example, consider the $C^\infty$ function
\[
f(x) =
\begin{cases}
    e^{-1/x^2}, & x \neq 0, \\
    0, & x = 0.
\end{cases}
\]
Near $0$,
\[
f'(x) = \frac{2}{x^3} f(x),
\]
so $\left|f'(x)\right| \geq cf(x) $ locally, which corresponds to $\theta = 1$. But for every $\theta < 1$, the ratio $\left|f'(x)\right| / f(x)^\theta$ tends to $0$ as $x \to 0$, so no exponent $\theta < 1$ works.

\citet{Bolt07} show that the K{\L} inequality in~\eqref{eq:KL} holds locally on $U$ for the very broad class of subanalytic functions. In what follows, as in, e.g., \citet{Fatk22}, we assume that~\eqref{eq:KL} holds globally. Assuming instead the local K{\L} property of \citet{Bolt07} would only require an initial burn-in phase and a localization argument, which we omit here to simplify the exposition and focus on final convergence rates and adaptation.

When $x$ is random (e.g., an iterate of a stochastic optimization method), we obtain the following result.

\begin{lemma}\label{lem:SKL}
Suppose that $f$ satisfies the K{\L} inequality~\eqref{eq:KL} globally with exponent $\theta\in[1/2,1]$, and that $x$ is a $\reals^d$-valued random variable supported on $U$. Then
\BEQ\label{eq:SKL}
\Expect[ \operatorname{dist}\!\left(0,\partial f(x)\right)^2 ]^{1/2} \geq c \left(\Expect[ f(x) ]-f^* \right)^\theta.
\EEQ
\end{lemma}
\begin{proof} Squaring~\eqref{eq:KL} and taking expectations yields $\Expect[ \operatorname{dist}\!\left(0,\partial f(x)\right)^2 ] \geq c^2 \Expect[(f(x)-f^*)^{2\theta}]$. For $\theta \geq 1/2$, the map $t\mapsto t^{2\theta}$ is convex on $\reals_+$, so Jensen's inequality gives 
\[
\Expect[ (f(x)-f^*)^{2\theta}] \geq (\Expect[f(x)]-f^*)^{2\theta},
\]
hence the result.
\end{proof}

Note that a similar argument does not hold for $\Expect[ \operatorname{dist}\!\left(0,\partial f(x)\right) ]$, as the following simple counterexample shows. Pick $f(x)=x^2/2$, so that $f^*=0$ and $\nabla f(x) =x$. The K{\L} inequality reads $|x| \geq c (x^2/2)^{1/2} = \frac{c}{\sqrt{2}}|x|$ and holds for $c\leq\sqrt{2}$. The stochastic version would read
\[
\Expect[|x|] \geq \sqrt{2} \Expect[x^2/2]^{1/2} = \Expect[x^2]^{1/2}.
\]
Now, setting $\Prob(x=a)=\epsilon$ for some $a>0$ and $\Prob(x=0)=1-\epsilon$, we get $\Expect[|x|]=\epsilon a$ and $\Expect[x^2]^{1/2}=a\sqrt{\epsilon}$, hence
\[
\frac{\Expect[|x|]}{\Expect[x^2]^{1/2}}=\sqrt{\epsilon},
\]
which means that we cannot have $\Expect[|x|] \geq c \Expect[x^2]^{1/2} $ for any $c>0$.

\section{Restart schemes \& the proximal stochastic subgradient method}\label{sec:restart}
We now focus on the following problem
\BEQ\label{eq:main-pb}
\min f(x) \triangleq g(x) + h(x)
\EEQ
in the variable $x\in\reals^d$ where the function $g(x)$ is $\rho$-weakly convex (i.e., $g(x)+\rho\|x\|^2/2$ is convex) and $h(x)$ is a convex proximable function, i.e., the proximal map
\BEQ\label{eq:prox}
\mathrm{prox}_{\alpha h}(x) \triangleq \argmin_y \left\{h(y) + \frac{1}{2\alpha}\|x-y\|^2\right\}
\EEQ
can be computed efficiently.

\subsection{The proximal stochastic subgradient method}
Defining the Moreau envelope as
\[
f_\lambda(x) \triangleq \min_{y} \left\{ f(y) + \frac{1}{2\lambda} \|y - x\|^2 \right\},    
\]
where $\lambda > 0$. Since $f$ is $\rho$-weakly convex, when $\lambda < \rho^{-1}$ the envelope $f_\lambda$ is $C^1$, with gradient \citep[Corollary~3.4]{Hohe20}
\[
\nabla f_\lambda(x) = \lambda^{-1}(x - \operatorname{prox}_{\lambda f}(x)).
\]
Furthermore, when $\lambda \leq (2\rho)^{-1}$, the gradient of the Moreau envelope is $\lambda^{-1}$-Lipschitz \citep[Corollary~3.4(b)]{Hohe20}, i.e.,
\BEQ\label{eq:moreau-smoothness}
\| \nabla f_\lambda (y) - \nabla f_\lambda (x) \| \leq \lambda^{-1} \| y - x \|.
\EEQ
As in \citet{Drus19,Drus17,Davi18}, we use the norm of the gradient $\|\nabla f_\lambda(x)\|$ as a measure of near-stationarity for problem~\eqref{eq:main-pb}, using the fact that the proximal point $\hat{x} \triangleq \operatorname{prox}_{\lambda f}(x)$ satisfies
\BEQ\label{eq:near-stat}
\begin{cases}
\|\hat{x} - x\| = \lambda \|\nabla f_\lambda(x)\|, \\
f(\hat{x}) \leq f(x), \\
\operatorname{dist}(0,\partial f(\hat{x})) \leq \|\nabla f_\lambda(x)\|,
\end{cases}
\EEQ
so a small gradient $\|\nabla f_\lambda(x)\|$ means that $x$ is close to some point $\hat x$ that is nearly stationary.

As in \citet{Davi18}, we assume that $G(x,\xi)$ is an unbiased estimator of a subgradient of $g$ satisfying
\[
\Expect\!_\xi\!\left[G(x,\xi)\right]\in\partial g(x) \quad \mbox{and} \quad \Expect\!_\xi[\|G(x,\xi)\|^2]\leq L^2
\]
for some $L>0$ and all $x \in \dom (h)$. Here, $\partial g(x)$ denotes the Fr\'echet subdifferential of $g$ at $x$, i.e., the set of all vectors $v$ satisfying
\[
g(y) \geq g(x) + \langle v, y - x \rangle + o(\|y - x\|) \quad \text{as } y \to x.
\]
In fact, weak convexity guarantees that the subgradients of $g$ satisfy the stronger property
\[
g(y) \geq g(x) + \langle v, y - x \rangle - \frac{\rho}{2} \|y - x\|^2,
\quad \forall x, y \in \mathcal{D}, \; v \in \partial g(x).
\]

\begin{algorithm}[H]
	\caption{Proximal subgradient}\label{algo:prox-subgrad}
	\begin{algorithmic}
		\State{\textbf{Inputs:} $x_0\in\reals^d$, a step size sequence $(\alpha_t)$, and an iteration budget $T$.}
		\For{$t=0,\ldots,T-1 $}
		\State Sample $\xi_t$ and set
		\vspace*{-0.3cm}
		\[
		x_{t+1} \triangleq \mathrm{prox}_{\alpha_{t} h}(x_{t} - \alpha_t G(x_{t},\xi_{t}))
		\]
		\vspace*{-0.5cm} 
		\EndFor
        \State Sample $t^* \in \{0, \ldots, T\}$ according to the probability distribution
            \[
            \Prob(t^* = t) = \frac{\alpha_t}{\sum_{t=0}^{T} \alpha_t}.
            \]
        \State{\textbf{Output:} $x_{t^*}$}
	\end{algorithmic}
\end{algorithm}

Given a step size sequence $\alpha_t$, the proximal stochastic subgradient method is detailed in \Cref{algo:prox-subgrad}. \citet[Cor.\,2.6]{Davi18} then show the following complexity result.

\begin{theorem}[Proximal stochastic subgradient method complexity]\label{th:complexity}
Fix a constant $\beta \in \left(0, \frac{1}{2\rho}\right]$, $T > 0$, and set the step size to $\alpha_t = {\beta}/{\sqrt{T+1}}$ for all $t$. Then the point $x_{t^*}$ returned by \Cref{algo:prox-subgrad} satisfies
\[
\Expect[\|\nabla f_{1/(2\rho)}(x_{t^*})\|^2]
\leq
\frac{2(f_{1/(2\rho)}(x_0)-f^*)}{\beta \sqrt{T+1}}
+ \frac{2\rho L^2 \beta}{\sqrt{T+1}}.
\]
Moreover, if $h$ is the indicator of a closed convex set $\mathcal{X}\subseteq\reals^d$, so that \Cref{algo:prox-subgrad} reduces to the projected stochastic subgradient method, the bound above holds for any value $\beta>0$.
\end{theorem}
\begin{proof}
See \citet[Cor.\,2.2 and 2.6]{Davi18}.
\end{proof}

\subsection{K{\L} inequality on the Moreau envelope}\label{s:KL-envelope}
Our results assume that a K{\L} inequality holds on the Moreau envelope $f_\lambda$ rather than on $f$ itself. Recall that the restriction $\theta\in[1/2,1]$ comes from the stochastic K{\L} inequality in~\eqref{eq:SKL}, where Jensen's inequality requires $\theta \geq 1/2$ to pass from the pointwise K{\L} inequality to a bound on $\Expect[\|\nabla f_\lambda(x)\|^2]^{1/2}$. We now show that for $\theta\in[1/2,1]$, the K{\L} inequality on $f_\lambda$ is in fact equivalent to that on~$f$.

Writing $\hat{x} = \operatorname{prox}_{\lambda f}(x)$, the definition of the Moreau envelope yields the identity
\BEQ\label{eq:envelope-identity}
f_\lambda(x) - f^* = \left(f(\hat{x}) - f^*\right) + \frac{\lambda}{2}\|\nabla f_\lambda(x)\|^2.
\EEQ
Dropping the first (nonnegative) term gives the upper bound
\BEQ\label{eq:envelope-upper}
\|\nabla f_\lambda(x)\| \leq \sqrt{\frac{2}{\lambda}}\,(f_\lambda(x) - f^*)^{1/2}.
\EEQ
Now suppose that $f_\lambda$ satisfies the K{\L} lower bound $\|\nabla f_\lambda(x)\| \geq c(f_\lambda(x) - f^*)^\theta$. Combining it with~\eqref{eq:envelope-upper} requires
\[
c\,(f_\lambda(x) - f^*)^\theta \leq \sqrt{\frac{2}{\lambda}}\,(f_\lambda(x) - f^*)^{1/2}
\]
for all $x$ near the optimal set. Dividing both sides by $(f_\lambda(x)-f^*)^{1/2}$ yields $c\,(f_\lambda(x) - f^*)^{\theta-1/2} \leq \sqrt{2/\lambda}$, which fails as $f_\lambda(x) - f^* \to 0$ whenever $\theta < 1/2$. Hence $f_\lambda$ can only be K{\L} with exponent $\theta \geq 1/2$. In particular, even if $f$ itself satisfies the K{\L} inequality with some exponent $\theta_f < 1/2$, the envelope $f_\lambda$ can only satisfy it with exponent at least $1/2$. This case is nevertheless covered by our results: since $(f(x)-f^*)^{\theta_f} \geq (f(x)-f^*)^{1/2}$ whenever $f(x)-f^* \leq 1$, $f$ then also satisfies~\eqref{eq:KL} with exponent $1/2$ on that region, and \Cref{prop:KL-transfer} below, applied with $\nu=1$, shows that $f_\lambda$ is K{\L} with exponent $1/2$ on $\{x : f_\lambda(x)-f^* \leq 1\}$, so our results apply with $\theta=1/2$. As we show below, within the admissible range $\theta\in[1/2,1]$, the K{\L} property transfers in both directions between $f$ and $f_\lambda$.

\begin{proposition}\label{prop:KL-transfer}
Suppose $f$ is $\rho$-weakly convex, $\lambda < \rho^{-1}$, $\theta\in[1/2,1]$, $c>0$ and $\nu>0$.\\
(i) If $\operatorname{dist}(0,\partial f(y)) \geq c(f(y)-f^*)^\theta$ for all $y$ such that $f(y)-f^*\leq\nu$, then $\|\nabla f_\lambda(x)\| \geq c_\nu(f_\lambda(x)-f^*)^\theta$ for all $x$ such that $f_\lambda(x)-f^*\leq\nu$, where
\BEQ\label{eq:KL-transfer-constant}
c_\nu \triangleq \left(c^{-1/\theta} + \frac{\lambda}{2}\left(\frac{2\nu}{\lambda}\right)^{1-1/(2\theta)}\right)^{-\theta}.
\EEQ
(ii) Conversely, if $\|\nabla f_\lambda(x)\| \geq c(f_\lambda(x)-f^*)^\theta$ for all $x\in\reals^d$, then $\operatorname{dist}(0,\partial f(y)) \geq c(f(y)-f^*)^\theta$ for all $y\in\dom\partial f$, with the same constant $c$.
\end{proposition}
\begin{proof}
(i)~%
Fix $x$ such that $f_\lambda(x)-f^*\leq\nu$, and write and $\delta = f(\hat{x})-f^*$. The proximal optimality condition gives $\nabla f_\lambda(x) \in \partial f(\hat{x})$, and~\eqref{eq:envelope-identity} gives $\delta \leq f_\lambda(x)-f^* \leq \nu$, so the K{\L} inequality on $f$ at $\hat{x}$ yields
\[
\|\nabla f_\lambda(x)\| \geq \operatorname{dist}(0,\partial f(\hat{x})) \geq c\,\delta^{\theta},
\]
hence $\delta \leq (\frac{1}{c} \|\nabla f_\lambda(x)\|)^{1/\theta}$. Moreover,~\eqref{eq:envelope-upper} gives $\|\nabla f_\lambda(x)\| \leq (2\nu/\lambda)^{1/2}$, so that $\|\nabla f_\lambda(x)\|^2 = \|\nabla f_\lambda(x)\|^{2-1/\theta}\|\nabla f_\lambda(x)\|^{1/\theta} \leq (2\nu/\lambda)^{1-1/(2\theta)}\|\nabla f_\lambda(x)\|^{1/\theta}$ since $\theta \geq 1/2$. The identity~\eqref{eq:envelope-identity} then gives
\[
f_\lambda(x) - f^* = \delta + \frac{\lambda}{2}\|\nabla f_\lambda(x)\|^2 \leq \left(c^{-1/\theta} + \frac{\lambda}{2}\left(\frac{2\nu}{\lambda}\right)^{1-1/(2\theta)}\right)\|\nabla f_\lambda(x)\|^{1/\theta} = (\frac{1}{c_\nu}\|\nabla f_\lambda(x)\|)^{1/\theta},
\]
which is the desired inequality. In particular, for $\theta=1/2$, the constant $c_\nu = (c^{-2}+\lambda/2)^{-1/2}$ does not depend on $\nu$.\\
(ii)~%
Fix $y\in\dom\partial f$ and $v\in\partial f(y)$, and set $x = y + \lambda v$. The optimality condition for $\operatorname{prox}_{\lambda f}(x)$ at $z$ reads $(x-z)/\lambda \in \partial f(z)$, which holds at $z=y$ since $(x-y)/\lambda = v \in \partial f(y)$. The proximal objective is $(1/\lambda - \rho)$-strongly convex when $\lambda < \rho^{-1}$, so $\operatorname{prox}_{\lambda f}(x) = y$ and $\nabla f_\lambda(x) = v$. The identity~\eqref{eq:envelope-identity} then gives $f_\lambda(x) - f^* = (f(y)-f^*) + \frac{\lambda}{2}\|v\|^2 \geq f(y)-f^*$, so the K{\L} inequality on $f_\lambda$ at $x$ yields
\[
\|v\| = \|\nabla f_\lambda(x)\| \geq c\left(f_\lambda(x)-f^*\right)^\theta \geq c\left(f(y)-f^*\right)^\theta.
\]
Since $v\in\partial f(y)$ was arbitrary, this gives $\operatorname{dist}(0,\partial f(y)) \geq c(f(y)-f^*)^\theta$.
\end{proof}

For $\theta\in[1/2,1)$, the implication in \Cref{prop:KL-transfer}(i) also follows from \citet[Th.\,5.2 and Remark\,5.1(i)]{Yu22}, who show that $f_\lambda$ is K{\L} with exponent $\theta$ for any $\lambda>0$ whenever $f$ is bounded below and K{\L} with exponent $\theta$ at every point of $\dom\partial f$. Our result also covers $\theta=1$ and the converse implication in \Cref{prop:KL-transfer}(ii), and yields explicit constants.

We now provide a lemma comparing the optimality gaps of $f$ and of its Moreau envelope, and show that the gradient of the envelope is bounded on $\mathcal{X}$. 

\begin{lemma}\label{lem:gap-envelope-bound}
    Suppose that $h$ is the indicator of a closed convex set $\mathcal{X}\subseteq\reals^d$ and let $\lambda = 1/(2\rho)$. Then, for all $x \in \mathcal{X}$,
    \BEQ\label{eq:moreau-gradient-bounded}
        \|\nabla f_\lambda(x)\| \leq 4L
    \EEQ
    and
    \BEQ\label{eq:gap-envelope-bound}
        f_{\lambda}(x)-f^* \;\leq\; f(x)-f^* \;\leq\; \left(f_{\lambda}(x)-f^*\right) + \frac{L}{\sqrt{\rho}}\left(f_{\lambda}(x)-f^*\right)^{1/2}.
    \EEQ
\end{lemma}
\begin{proof}
Fix $x\in\mathcal{X}$ and let $\hat x = \operatorname{prox}_{\lambda f}(x)$, so that $\hat x \in \mathcal{X}$ since $f(\hat x) \leq f_\lambda(x) < +\infty$. Moreover, $\|x - \hat x\| = \lambda\|\nabla f_\lambda(x)\|$ by~\eqref{eq:near-stat}. Let $v = \Expect_\xi[G(x,\xi)] \in \partial g(x)$. By Jensen's inequality, $\|v\| \leq \Expect_\xi[\|G(x,\xi)\|^2]^{1/2} \leq L$. Since $x,\hat x\in\mathcal{X}$, we have $f(x) - f(\hat x) = g(x) - g(\hat x)$, and the weak convexity of $g$ gives
\[
g(\hat x) \geq g(x) + \langle v, \hat x - x\rangle - \frac{\rho}{2}\| x - \hat x\|^2.
\]
Hence, by the Cauchy--Schwarz inequality,
\BEQ\label{eq:prox-step-estimate}
f(x) - f(\hat x) \leq L\| x - \hat x\| + \frac{\rho}{2}\| x - \hat x\|^2.
\EEQ
On the other hand, evaluating the objective defining $\operatorname{prox}_{\lambda f}(x)$ at $y=x$ and at $y=\hat x$ gives $f(\hat x) + \frac{1}{2\lambda}\| x - \hat x\|^2 \leq f(x)$, that is, $\frac{1}{2\lambda}\| x - \hat x\|^2 \leq f(x)-f(\hat x)$. Combining this with~\eqref{eq:prox-step-estimate} and $\lambda = 1/(2\rho)$, we get $\rho\| x - \hat x\|^2 \leq L\| x - \hat x\| + \frac{\rho}{2}\| x - \hat x\|^2$, hence $\| x - \hat x\| \leq 2L/\rho$ and
\[
\|\nabla f_\lambda(x)\| = \lambda^{-1}\| x - \hat x\| = 2\rho\| x - \hat x\| \leq 4L,
\]
which is~\eqref{eq:moreau-gradient-bounded}. The left-hand inequality in~\eqref{eq:gap-envelope-bound} follows from the definition of the Moreau envelope, evaluated at $y=x$. For the right-hand one, rewrite~\eqref{eq:prox-step-estimate} using $\|x - \hat x\| = \lambda\|\nabla f_\lambda(x)\|$ as
\[
f(x) - f(\hat x) \leq L\lambda\|\nabla f_\lambda(x)\| + \frac{\rho\lambda^2}{2}\|\nabla f_\lambda(x)\|^2 = \frac{L}{2\rho}\|\nabla f_\lambda(x)\| + \frac{\lambda}{4}\|\nabla f_\lambda(x)\|^2.
\]
Adding $f(\hat x) - f^*$ to both sides, bounding $\frac{\lambda}{4}\|\nabla f_\lambda(x)\|^2 \leq \frac{\lambda}{2}\|\nabla f_\lambda(x)\|^2$ and using the identity~\eqref{eq:envelope-identity}, we get
\[
f(x) - f^* \leq \left(f(\hat x) - f^*\right) + \frac{\lambda}{2}\|\nabla f_\lambda(x)\|^2 + \frac{L}{2\rho}\|\nabla f_\lambda(x)\| = \left(f_\lambda(x) - f^*\right) + \frac{L}{2\rho}\|\nabla f_\lambda(x)\|.
\]
Finally,~\eqref{eq:envelope-upper} gives $\|\nabla f_\lambda(x)\| \leq 2\sqrt{\rho}\,(f_\lambda(x)-f^*)^{1/2}$, hence the result.
\end{proof}

The bound~\eqref{eq:moreau-gradient-bounded} combined with the K{\L} inequality, allows us to bound the optimality gap of the function $f$ and its Moreau envelope on $\mathcal{X}$.
\begin{corollary}\label{cor:gap-bounded}
    Suppose that $h$ is the indicator of a closed convex set $\mathcal{X}\subseteq\reals^d$ and that $f_{1/(2\rho)}$ satisfies the K{\L} inequality~\eqref{eq:KL} globally for some $\theta\in[1/2,1]$ and $c>0$. Then, for all $x\in\mathcal{X}$,
    \BEQ\label{eq:gap-bounded}
        f_{1/(2\rho)}(x)-f^* \;\leq\; \left(\frac{4L}{c}\right)^{1/\theta}
        \qquad\mbox{and}\qquad
        f(x)-f^* \;\leq\; \left(\frac{4L}{c}\right)^{1/\theta} + \frac{L}{\sqrt{\rho}}\left(\frac{4L}{c}\right)^{1/(2\theta)}.
    \EEQ
\end{corollary}
\begin{proof}
    Let $x\in\mathcal{X}$, write $\lambda=1/(2\rho)$ and $\delta \triangleq f_\lambda(x)-f^*\geq0$. If $\delta=0$, both inequalities in~\eqref{eq:gap-bounded} hold trivially, so assume $\delta>0$. The envelope $f_\lambda$ is differentiable, so $\operatorname{dist}(0,\partial f_\lambda(x)) = \|\nabla f_\lambda(x)\|$ and the K{\L} inequality~\eqref{eq:KL} for $f_\lambda$ reads $\|\nabla f_\lambda(x)\| \geq c\,\delta^\theta$. Combining it with~\eqref{eq:moreau-gradient-bounded} gives
    \[
        c\,\delta^\theta \;\leq\; \|\nabla f_\lambda(x)\| \;\leq\; 4L,
        \qquad\mbox{that is,}\qquad \delta^\theta \leq \frac{4L}{c}.
    \]
    Since $\theta>0$, the map $u\mapsto u^{1/\theta}$ is increasing on $\reals_+$, and raising both sides to the power $1/\theta$ yields $\delta \leq (4L/c)^{1/\theta}$, which is the first bound. The map $\delta\mapsto\delta+\frac{L}{\sqrt{\rho}}\delta^{1/2}$ is also increasing on $\reals_+$, so substituting this bound into the right-hand inequality of~\eqref{eq:gap-envelope-bound} gives the second one.
\end{proof}

\subsection{Restart schemes}
Suppose now that we restart \Cref{algo:prox-subgrad}, setting a new step size at each outer iteration. Let us denote by $x_k$ the iterate produced by the restart scheme after $k$ outer iterations. Let us also assume that $f_\lambda$ satisfies the K{\L} inequality
\[
\|\nabla f_\lambda(x)\| \geq c(f_\lambda(x)-f^*)^\theta, \quad \mbox{for } x\in V,
\]
for some $\theta\in[1/2,1]$ (recall that $f_\lambda^*=f^*$). As discussed in \Cref{lem:SKL}, this also implies 
\begin{equation}
  \label{eq:kl-moreau}
\Expect[\|\nabla f_\lambda(x)\|^2]^{1/2} \geq c \left(\Expect[f_\lambda(x)]-f^* \right)^\theta.
\end{equation}


We start the outer iterations of the restart scheme at a point $x_0\in\reals^d$ and write $\Delta_0=f_\lambda(x_0)-f^*$. We will design these iterations to ensure
\BEQ\label{eq:delta-k}
\Expect[f_\lambda(x_{k})]-f^* \leq \Delta_0 \exp(-\gamma k)
\EEQ
for $k\geq 0$ and some $\gamma>0$. Let $t_k$ be the number of inner iterations at outer iteration $k$. We have the following result.

\begin{lemma}\label{lem:step}
    Suppose $h$ is the indicator of a closed convex set $\mathcal{X}\subseteq\reals^d$ and $f_{1/(2\rho)}$ satisfies the K{\L} inequality in~\eqref{eq:KL} globally for some $\theta\in[1/2,1]$. Running \Cref{algo:prox-subgrad} with step size
    \BEQ\label{eq:lr}
        \alpha_k = \frac{(f_{1/(2\rho)}(x_k)-f^*)^{1/2}}{\sqrt{\rho} L \sqrt{t_k+1}}
    \EEQ
    for 
    \BEQ\label{eq:tk}
        t_{k} \geq 16 c^{-4} \rho L^2 \Delta_0^{1-4\theta} \exp(4\theta \gamma) \exp((4\theta-1)\gamma k)
    \EEQ
    iterations at each outer iteration $k$ ensures $\Expect[f_{1/(2\rho)}(x_{k})]-f^* \leq \Delta_0 \exp(-\gamma k)$.
\end{lemma}
\begin{proof}
We show by induction on $k$ that $\Expect[f_{1/(2\rho)}(x_{k})]-f^* \leq \Delta_0 \exp(-\gamma k)$ for every $k\geq0$. For $k=0$, this holds with equality, since $x_0$ is deterministic. Suppose now that it holds at some $k\geq0$, and let $x_{k+1}$ be the output of \Cref{algo:prox-subgrad} initialized at $x_k$.
Setting 
\[
\beta = \frac{(f_{1/(2\rho)}(x_k)-f^*)^{1/2}}{\sqrt{\rho} L}
\]
in \Cref{th:complexity}, we get
\[
\Expect[\|\nabla f_{1/(2\rho)}(x_{k+1})\|^2] \leq \frac{4\sqrt{\rho}L(f_{1/(2\rho)}(x_k)-f^*)^{1/2}}{\sqrt{t_k+1}}.
\]
Since $x \mapsto x^{1/2}$ is concave, taking the expectation and using Jensen's inequality gives
\[
\Expect[\|\nabla f_{1/(2\rho)}(x_{k+1})\|^2]
\leq \Expect\!\left[\frac{4\sqrt{\rho}L(f_{1/(2\rho)}(x_k)-f^*)^{1/2}}{\sqrt{t_k+1}}\right]
\leq \frac{4\sqrt{\rho}L(\Expect[f_{1/(2\rho)}(x_{k})]-f^*)^{1/2}}{\sqrt{t_k+1}}.
\]
If $f_{1/(2\rho)}$ satisfies~\eqref{eq:KL} for some $\theta\in[1/2,1]$, we have, as in~\eqref{eq:SKL},
\[
\Expect[f_{1/(2\rho)}(x_{k+1})]-f^* \leq c^{-1/\theta} \Expect[\|\nabla f_{1/(2\rho)}(x_{k+1})\|^2]^{1/(2\theta)}
\]
and we then need to ensure
\[
\Expect[f_{1/(2\rho)}(x_{k+1})]-f^* \leq c^{-1/\theta} \left(\frac{4\sqrt{\rho}L\left(\Delta_0\exp(-\gamma k) \right)^{1/2}}{\sqrt{t_k+1}}\right)^{1/(2\theta)} \leq \Delta_0 \exp(-\gamma (k+1))
\]
which holds as soon as
\[
t_k +1 \geq 16 c^{-4} \rho L^2 \Delta_0^{1-4\theta} \exp(4\theta \gamma) \exp((4\theta-1)\gamma k).
\]
This is guaranteed by~\eqref{eq:tk}, which completes the induction step and hence the proof.
\end{proof}

We can use the bound in \Cref{lem:step} to get a global convergence rate for the restart scheme in \Cref{algo:prox-subgrad-restart}.

\begin{algorithm}[H]
	\caption{Restarted proximal subgradient}\label{algo:prox-subgrad-restart}
		\begin{algorithmic}
			\State{\textbf{Inputs:} $x_0\in\reals^d,~\gamma>0,~N\in\mathbb{N}$.}
			\For{$k=0,\ldots,N-1 $}
			\State Run \Cref{algo:prox-subgrad}, starting at $x_{k}$, for 
	        \[
	            t_{k} = 16 c^{-4} \rho L^2 \Delta_0^{1-4\theta} \exp(4\theta \gamma) \exp((4\theta-1)\gamma k)
	        \]
	        iterations, with step size
	        \[
	            \alpha_k = \frac{(f_{1/(2\rho)}(x_k)-f^*)^{1/2}}{\sqrt{\rho} L \sqrt{t_k+1}}
	        \]
            and set $x_{k+1}$ as its output.
			\EndFor
	        \State{\textbf{Output:} $x_{N}$}
		\end{algorithmic}
\end{algorithm}

\begin{theorem}\label{th:primal-gap-complexity-bound}
    Suppose $h$ is the indicator of a closed convex set $\mathcal{X}\subseteq\reals^d$, so that \Cref{algo:prox-subgrad} reduces to the projected stochastic subgradient method, and that $f_{1/(2\rho)}$ satisfies the K{\L} inequality in~\eqref{eq:KL} globally for some $\theta\in[1/2,1]$. Let $T \triangleq \sum_{k=0}^{N-1} t_k$ be the total inner-iteration budget after $N$ outer iterations of the restart scheme in \Cref{algo:prox-subgrad-restart}. Then
    \BEQ\label{eq:restart-complexity}
        \Expect[f_{1/(2\rho)}(x_N)]-f^* \leq \frac{f_{1/(2\rho)}(x_0)-f^*}{\left(\gamma(4\theta-1)C^{-1} T + 1\right)^{1/(4 \theta -1)}}
    \EEQ
    where $C=16 c^{-4} \rho L^2 \Delta_0^{1-4\theta} \exp(4\theta \gamma)$.
\end{theorem}
\begin{proof}
    \Cref{lem:step} shows that we need to set 
    \[
        t_{k} \geq 16 c^{-4} \rho L^2 \Delta_0^{1-4\theta} \exp(4\theta \gamma) \exp((4\theta-1)\gamma k)
    \]
    to ensure $\Expect[f_{1/(2\rho)}(x_{k})]-f^* \leq \Delta_0 \exp(-\gamma k)$. \citet[Lem.\,2.1]{Roul17} then yields the desired result.
\end{proof}

Although the asymptotic convergence rate is independent of $\gamma > 0$, the previous bound is minimal for $\gamma = \frac{1}{4\theta}$. In practice, we can simply pick $\gamma=2$. We can derive a similar result in terms of $\Expect[\|\nabla f_{1/(2\rho)}(x)\|^2]$, which controls near-stationarity as detailed in~\eqref{eq:near-stat}.

\begin{theorem}\label{th:restart-stationarity}
Suppose $h$ is the indicator of a closed convex set $\mathcal{X}\subseteq\mathbb{R}^d$, so that \Cref{algo:prox-subgrad} reduces to the projected stochastic subgradient method, and that $f_{1/(2\rho)}$ satisfies the K{\L} inequality~\eqref{eq:KL} globally for some $\theta\in[1/2,1]$. Let $T \triangleq \sum_{k=0}^{N-1} t_k$ be the total inner-iteration budget after $N$ outer iterations of the restart scheme in \Cref{algo:prox-subgrad-restart}. Then
    \begin{equation}\label{eq:convergence-guarantee}
        \Expect[\|\nabla f_{1/(2\rho)}(x_N)\|^2]^{1/2}
        \leq
        \frac{\|\nabla f_{1/(2\rho)}(x_0)\|}{\left( \gamma(4\theta-1)C^{-1} T + 1 \right)^{\theta/(4\theta - 1)}},
    \end{equation}
    where $C=16 c^{-4} \rho L^2 \Delta_0^{1-4\theta} \exp(4\theta \gamma)$.
\end{theorem}
\begin{proof} Let $k \in \left\lbrace 0,\ldots, N-1 \right\rbrace$. As in the proof of \Cref{lem:step} above, running \Cref{algo:prox-subgrad} from $x_k$, with
\[
\beta = \frac{\left(f_{1/(2\rho)}(x_k)-f^*\right)^{1/2}}{\sqrt{\rho} L}
\]
in \Cref{th:complexity}, we get
\[
    \Expect[\|\nabla f_{1/(2\rho)}(x_{k+1})\|^2] \leq \frac{4\sqrt{\rho}L\left(\Expect[f_{1/(2\rho)}(x_{k})]-f^*\right)^{1/2}}{\sqrt{t_k+1}}.
\] 
By \Cref{lem:step}, setting
\[
    t_k \geq 16 c^{-4} \rho L^2 \Delta_0^{1-4\theta} \exp(4\theta \gamma) \exp((4\theta-1)\gamma k)
\]
ensures $\Expect[f_{1/(2\rho)}(x_{k})]-f^* \leq \Delta_0 \exp(-\gamma k)$, hence
\begin{align*}
    \Expect[\|\nabla f_{1/(2\rho)}(x_{k+1})\|^2]
    & \leq \frac{4\sqrt{\rho}L\left(\Delta_0 \exp(-\gamma k) \right)^{1/2}}{(16 c^{-4} \rho L^2 \Delta_0^{1-4\theta} \exp(4\theta \gamma) \exp((4\theta-1)\gamma k))^{1/2}} \\
    & \leq c^2 \Delta_0^{2\theta} \exp{(-2\theta\gamma -2\theta\gamma k)}.
\end{align*}
Using the K{\L} inequality~\eqref{eq:KL}, we have $c\Delta_0^\theta \leq \|\nabla f_{1/(2\rho)}(x_0)\|$.
This gives
\[
    \Expect[\|\nabla f_{1/(2\rho)}(x_{k+1})\|^2]^{1/2} \leq \|\nabla f_{1/(2\rho)}(x_0)\| \exp{(-\theta\gamma (k+1))}.
\]
Again, \citet[Lem.\,2.1]{Roul17} yields the desired result.
\end{proof}

We compare these rates with existing ones in \Cref{sec:rate-comparison}, after showing in \Cref{sec:minimax-optimality} that they cannot be improved.

When the optimality gap of the Moreau envelope is not available at each outer iteration, we can define the step size using only $\Delta_0$, as detailed in the following result.

\begin{lemma}\label{lem:exp-lr}
    Suppose $h$ is the indicator of a closed convex set $\mathcal{X}\subseteq\mathbb{R}^d$ and $f_{1/(2\rho)}$ satisfies the K{\L} inequality \labelcref{eq:KL} globally for some $\theta\in[1/2,1]$. Running \Cref{algo:prox-subgrad} with step size
    \begin{equation}\label{eq:exp-lr}
        \alpha_k = \frac{\Delta_0^{1/2} \exp{(-\gamma k/2)}}{\sqrt{\rho} L \sqrt{t_k+1}}
    \end{equation}
    for 
    \begin{equation}\label{eq:tk-exp}
        t_{k} \geq 16 c^{-4} \rho L^2 \Delta_0^{1-4\theta} \exp(4\theta \gamma) \exp((4\theta-1)\gamma k)
    \end{equation}
    iterations at each outer iteration $k$ ensures $\Expect[f_{1/(2\rho)}(x_{k})]-f^* \leq \Delta_0 \exp(-\gamma k)$.
\end{lemma}
\begin{proof}
We show by induction on $k$ that $\Expect[f_{1/(2\rho)}(x_{k})]-f^* \leq \Delta_0 \exp{(-\gamma k)}$ for every $k\geq0$. For $k=0$, this holds with equality, since $x_0$ is deterministic. Suppose now that it holds at some $k\geq0$, and let $x_{k+1}$ be the output of \Cref{algo:prox-subgrad} run from $x_k$ with
\[
\beta = \frac{\Delta_0^{1/2} \exp{(-\gamma k/2)}}{\sqrt{\rho} L}.
\]
Applying \Cref{th:complexity} and taking the expectation gives
\[
\Expect[\|\nabla f_{1/(2\rho)}(x_{k+1})\|^2] \leq \frac{2\sqrt{\rho}L}{\sqrt{t_k+1}}\left(\frac{\Expect[f_{1/(2\rho)}(x_{k})]-f^*}{\left(\Delta_0 \exp(-\gamma k)\right)^{1/2}}+\left(\Delta_0 \exp(-\gamma k)\right)^{1/2}\right).
\]
Assuming $f_{1/(2\rho)}$ satisfies~\eqref{eq:KL} for some $\theta\in[1/2,1]$, we have, as in~\eqref{eq:SKL},
\[
\Expect[f_{1/(2\rho)}(x_{k+1})]-f^* \leq c^{-1/\theta} \Expect[\|\nabla f_{1/(2\rho)}(x_{k+1})\|^2]^{1/(2\theta)}.
\]
Hence, for
\[
t_k +1 \geq 16 c^{-4} \rho L^2 \Delta_0^{1-4\theta} \exp(4\theta \gamma) \exp((4\theta-1)\gamma k),
\]
we have
\[
\Expect[f_{1/(2\rho)}(x_{k+1})]-f^* \leq c^{-1/\theta} \left(\frac{4\sqrt{\rho}L\Delta_0^{1/2} \exp{(-\gamma k/2)}}{\sqrt{t_k+1}}\right)^{1/(2\theta)} \leq \Delta_0 \exp(-\gamma (k+1)),
\]
which is the induction hypothesis at $k+1$ and completes the proof.
\end{proof}

This allows us to define a restart scheme that does not require the optimality gap at each outer iteration, namely a version of \Cref{algo:prox-subgrad-restart} with deterministic step sizes, detailed in \Cref{algo:prox-subgrad-restart-grad}. \Cref{lem:exp-lr} then yields bounds identical to those of \Cref{th:primal-gap-complexity-bound,th:restart-stationarity}.

\begin{algorithm}[H]
	\caption{Restarted proximal subgradient with deterministic step sizes}\label{algo:prox-subgrad-restart-grad}
		\begin{algorithmic}
			\State{\textbf{Inputs:} $x_0\in\mathbb{R}^d,~\gamma>0,~N\in\mathbb{N}$.}
			\For{$k=0,\ldots,N-1 $}
			\State Run \Cref{algo:prox-subgrad}, starting at $x_{k}$, for 
	        \[
	            t_{k} = 16 c^{-4} \rho L^2 \Delta_0^{1-4\theta} \exp(4\theta \gamma) \exp((4\theta-1)\gamma k)
	        \]
	        iterations, with step size
	        \[
	            \alpha_k = \frac{\left(\Delta_0 \exp(-\gamma k)\right)^{1/2}}{\sqrt{\rho} L \sqrt{t_k+1}}
	        \]
            and set $x_{k+1}$ as its output.
			\EndFor
	        \State{\textbf{Output:} $x_{N}$}
		\end{algorithmic}
\end{algorithm}

\subsection{Optimality}\label{sec:minimax-optimality}

We now show that the exponent in the convergence theorems above cannot be improved under the stochastic first-order oracle considered here.  The
argument follows from a reduction to the sharp lower bounds for sublevel-$\alpha$-Polyak--\L{}ojasiewicz functions in \citet{Masi26}.  The only additional ingredient is a comparison between the optimality gaps of the objective and of its Moreau envelope.

We use the standard local interpretation of the K{\L} inequality: it is required on an initial sublevel set, or more generally on a neighborhood containing the stochastic trajectory. 

\begin{theorem}[Minimax optimality]
\label{thm:minimax-optimality}
Fix $\theta\in[1/2,1]$ and set $\lambda=(2\rho)^{-1}$.  Consider randomized
projected methods that make $T$ calls to an unbiased stochastic first-order
oracle and then return $x_T$.  There is a subclass of the problems considered in
\Cref{th:primal-gap-complexity-bound}, consisting in fact of smooth objectives
on compact convex sets, for which $f_\lambda$ satisfies the K{\L} inequality with
exponent $\theta$ and for which
\begin{equation}
\label{eq:minimax-risk-lower}
 \inf_{\mathsf A}\sup_{(f,G)}
 \Expect\bigl[f_\lambda(x_T^{\mathsf A})-f^*\bigr]
 \geq
 \Omega_\theta\!\left(
     \left(\frac{\rho\sigma^2}{c^4T}\right)^{1/(4\theta-1)}
 \right)
\end{equation}
in the noise-dominated regime.  The supremum is over the preceding subclass
with fixed problem parameters, $c$ is the K{\L} constant, and $\sigma^2>0$ is the
variance level of the hard oracle.  As usual for first-order minimax lower
bounds, the ambient dimension is allowed to grow with the oracle budget.
Equivalently, any method that
guarantees $\Expect[f_\lambda(x_T)]-f^*\leq\epsilon$ requires
\begin{equation}
\label{eq:minimax-complexity-lower}
 T_\epsilon
 \geq
 \Omega_\theta\!\left(
     \rho\sigma^2c^{-4}\epsilon^{-(4\theta-1)}
 \right).
\end{equation}
Consequently, the rate
$\Expect[f_\lambda(x_T)]-f^*=\mathcal{O}(T^{-1/(4\theta-1)})$ is minimax optimal
in its dependence on $T$.
\end{theorem}

\begin{proof}
We first treat $\theta\in(1/2,1]$ and define
\[
    \alpha\triangleq\frac{1}{\theta}\in[1,2).
\]
For a differentiable objective $F$, the K{\L} inequality
\[
    \|\nabla F(x)\|\geq c_0\bigl(F(x)-F^*\bigr)^\theta
\]
is equivalent to the $\alpha$-PL inequality
\begin{equation}
\label{eq:alpha-pl-conversion}
    F(x)-F^*
    \leq \tau\|\nabla F(x)\|^\alpha,
    \qquad
    \tau=c_0^{-1/\theta}.
\end{equation}
\citet[Theorem\,4.1]{Masi26} show that, over the class of globally
$\ell$-smooth objectives satisfying~\eqref{eq:alpha-pl-conversion} on their
initial sublevel set, the bounded-variance stochastic first-order oracle
complexity obeys
\begin{equation}
\label{eq:masiha-lower-bound}
 T_\epsilon
 \geq
 \Omega_\alpha\!\left(
    \ell\tau^{2/\alpha}\epsilon^{-(2-\alpha)/\alpha}
    +
    \ell\sigma^2\tau^{4/\alpha}
          \epsilon^{-(4-\alpha)/\alpha}
 \right)
\end{equation}
whenever
\[
  0<\epsilon\leq a_\alpha
  \min\!\left\{\Delta_0,
       \ell^{\alpha/(2-\alpha)}\tau^{2/(2-\alpha)}\right\}
\]
for a constant $a_\alpha>0$.  Substituting
$\alpha=1/\theta$ and $\tau=c_0^{-1/\theta}$ yields
\begin{equation}
\label{eq:kl-objective-lower-bound}
 T_\epsilon
 \geq
 \Omega_\theta\!\left(
    \ell c_0^{-2}\epsilon^{-(2\theta-1)}
    +
    \ell\sigma^2c_0^{-4}\epsilon^{-(4\theta-1)}
 \right).
\end{equation}
The second term dominates as $\epsilon\downarrow0$ whenever $\sigma>0$.

It remains to transfer this lower bound from $F(x)-F^*$ to the criterion
$f_\lambda(x)-f^*$ used here.  Take a hard instance from \citet{Masi26} with smoothness
\begin{equation}
\label{eq:hard-smoothness-scaling}
    \ell=\frac{\rho}{2},
\end{equation}
and let $\mathcal X$ be a compact convex set whose interior contains the hard
instance's initial sublevel set and its minimizers.  Define
\[
    f\triangleq F+\iota_{\mathcal X}.
\]
Then $f$ is $\rho$-weakly convex.  Restricting the oracle to $\mathcal X$ can
only make the problem harder.  Moreover, because $\mathcal X$ is compact, the
smoothness of $F$ implies that $\nabla F$ is bounded on $\mathcal X$.
Together with the oracle variance bound, this yields
\begin{equation}
\label{eq:oracle-second-moment}
 \Expect\bigl[\|G(x,\xi)\|^2\bigr]
 =\|\nabla F(x)\|^2
  +\Expect\bigl[\|G(x,\xi)-\nabla F(x)\|^2\bigr]
 \leq L^2,
 \qquad x\in\mathcal X,
\end{equation}
after enlarging $L$ if necessary.  Hence these hard instances satisfy the
bounded-second-moment oracle assumption used in this paper.

For every $x\in\mathcal X$, the descent lemma gives, for all
$y\in\mathcal X$,
\[
 F(y)\geq F(x)+\langle\nabla F(x),y-x\rangle
             -\frac{\ell}{2}\|y-x\|^2.
\]
Since $\lambda^{-1}>\ell$, minimizing the right-hand side after adding
$\|y-x\|^2/(2\lambda)$ gives
\begin{equation}
\label{eq:envelope-lower-comparison-a}
 f_\lambda(x)
 \geq
 F(x)-\frac{\|\nabla F(x)\|^2}{2(\lambda^{-1}-\ell)}.
\end{equation}
Every lower-bounded $\ell$-smooth function satisfies
\begin{equation}
\label{eq:smooth-gradient-upper}
    \|\nabla F(x)\|^2
    \leq 2\ell\bigl(F(x)-F^*\bigr),
\end{equation}
as follows by applying the descent lemma at
$x-\ell^{-1}\nabla F(x)$.  Combining
\eqref{eq:envelope-lower-comparison-a} and
\eqref{eq:smooth-gradient-upper}, and using $f^*=F^*$, yields
\begin{equation}
\label{eq:envelope-gap-comparison}
 f_\lambda(x)-f^*
 \geq
 \frac{1-2\lambda\ell}{1-\lambda\ell}
       \bigl(F(x)-F^*\bigr).
\end{equation}
With $\lambda=(2\rho)^{-1}$ and $\ell=\rho/2$, we have
$\lambda\ell=1/4$, so in particular
\begin{equation}
\label{eq:two-thirds-comparison}
 f_{1/(2\rho)}(x)-f^*
 \geq \frac{2}{3}\bigl(F(x)-F^*\bigr).
\end{equation}

The K{\L} exponent is also preserved by passage to the Moreau envelope.  More
precisely, $f$ coincides with $F$ on the initial sublevel set $\{x : F(x)-F^*\leq\nu\}$ of the hard instance, which lies in the interior of $\mathcal X$, so $f$ satisfies the K{\L} inequality with exponent $\theta$ and constant $c_0$ on this set. \Cref{prop:KL-transfer}(i) then shows that $f_\lambda$ satisfies the K{\L} inequality with exponent $\theta$ on $\{x : f_\lambda(x)-f^*\leq\nu\}$, with the constant $c_\nu$ given by~\eqref{eq:KL-transfer-constant} for $c=c_0$. This constant differs from $c_0$ only by a
factor depending on $\theta$ and the fixed problem parameters.  We may
therefore replace $c_0$ with $c$ in~\eqref{eq:kl-objective-lower-bound} up to
the constants hidden in $\Omega_\theta(\cdot)$.

If an algorithm produced $x_T$ with expected envelope gap at most $\epsilon$,
then~\eqref{eq:two-thirds-comparison} would give expected $F$-gap at most
$3\epsilon/2$.  Applying~\eqref{eq:kl-objective-lower-bound} with
$\ell=\rho/2$ therefore proves~\eqref{eq:minimax-complexity-lower}, and
inverting that relation proves~\eqref{eq:minimax-risk-lower}.

At the endpoint $\theta=1/2$, we have $\alpha=2$.  \citet[Theorem\,4.3]{Masi26}
treat this case separately and give the
variance-dependent lower bound
\[
    T_\epsilon
    \geq \Omega\!\left(\ell\tau^2\sigma^2\epsilon^{-1}\right).
\]
Here $\tau=c_0^{-2}$, and the same Moreau-envelope comparison gives
$T_\epsilon=\Omega(\rho\sigma^2c^{-4}\epsilon^{-1})$.  This is exactly
\eqref{eq:minimax-complexity-lower} at $\theta=1/2$.
\end{proof}

\subsection{Comparison with existing rates}\label{sec:rate-comparison}
We now compare the rates of \Cref{th:primal-gap-complexity-bound,th:restart-stationarity} with existing ones. These comparisons are summarized in \Cref{tab:convergence-rates}.

Our restart scheme achieves the rate $\Expect[f_{1/(2\rho)}(x_N)]-f^* = \mathcal{O}(T^{-1/(4\theta - 1)})$. \citet{Font21,Fatk22} show that SGD achieves the same rate $\Expect[f(x_T)]-f^* = \mathcal{O}(T^{-1/(4\theta-1)})$ when $f$ is smooth, satisfies a global K{\L} inequality, and the stochastic gradients have bounded variance. \Cref{th:primal-gap-complexity-bound} thus shows that this rate does not require smoothness, with the caveat that it bounds the optimality gap of the Moreau envelope, while \citet{Font21,Fatk22} bound the optimality gap of the objective function itself.

\citet{Fatk22} also obtain the faster rate $\mathcal{O}(T^{-1/(2\theta)})$ for $\theta\in(1/2,1]$ with PAGER, which combines the variance-reduced PAGE estimator with restarts. Besides smoothness of $f$, PAGER requires an average smoothness condition on the stochastic gradients, which holds, e.g., when $\Expect_\xi[\|G(x,\xi)-G(y,\xi)\|^2]\leq \bar L^2\|x-y\|^2$ (a condition that, by Jensen's inequality, implies that $\nabla f$ is $\bar L$-Lipschitz). PAGER therefore does not apply to the nonsmooth setting considered here. Its faster rate does not contradict \Cref{thm:minimax-optimality} either: the hard instances there have smooth objectives, but their stochastic gradients are discontinuous in $x$ and do not satisfy average smoothness. Hence, even for smooth objectives, the rate $\mathcal{O}(T^{-1/(4\theta-1)})$ cannot be improved without additional assumptions on the stochastic gradients.

Compared to \citet[Cor.\,2.2]{Davi18} which, under weak convexity, gives the baseline rate $\Expect[\|\nabla f_{1/(2\rho)}(x_N)\|^2] = \mathcal{O}(T^{-1/2})$ for the projected stochastic subgradient method, the restart scheme exploits the K{\L} inequality to obtain the faster rate $\Expect[\|\nabla f_{1/(2\rho)}(x_N)\|^2] = \mathcal{O}(T^{-2\theta/(4\theta - 1)})$. Since $2\theta/(4\theta - 1) \in [2/3,1]$ for $\theta \in [1/2,1]$, this improves on the rate of \citet{Davi18} over the whole admissible range of K{\L} exponents, with the strongest acceleration at $\theta = 1/2$.

\section{Restarted proximal stochastic subgradient descent with momentum}\label{sec:momentum}

In this section, we consider the stochastic subgradient method with momentum studied by \citet{Mai20} and show that, under a K{\L} inequality, a restart scheme achieves the same convergence rates as without momentum.

\subsection{Preliminaries}
Before studying the restart scheme applied to the momentum method, we need to adapt the result of \citet{Mai20} to our setting. We begin by recalling the algorithm in the case where $h$ is the indicator function of a closed convex set $\mathcal{X} \subseteq \mathbb{R}^d$, given a step size $\alpha$ and a momentum parameter $\mu$.

\begin{algorithm}[H]
	\caption{Proximal subgradient with momentum}\label{algo:prox-subgrad-momentum}
	\begin{algorithmic}
		\State{\textbf{Inputs:} $x_0\in\mathbb{R}^d$, a step size $\alpha$, a momentum parameter $\mu$, and an iteration budget $T$.}
		\State{Sample $\xi_0$ and set $z_0 = G(x_0,\xi_0)$}
		\For{$t=0,\ldots,T-1 $}
		\State Set
		\vspace*{-0.3cm}
		\[
		x_{t+1} \triangleq \mathrm{proj}_{\mathcal{X}}(x_t - \alpha z_t)
		\]
		\State Sample $\xi_{t+1}$ and set
		\vspace*{-0.3cm}
		\[
		z_{t+1} \triangleq \mu G(x_{t+1}, \xi_{t+1}) + (1-\mu) \frac{x_t - x_{t+1}}{\alpha}
		\]
		\vspace*{-0.5cm} 
		\EndFor
        \State Sample $t^* \in \{0, \ldots, T\}$ uniformly.
        \State{\textbf{Output:} $x_{t^*}$}
	\end{algorithmic}
\end{algorithm}

Throughout this section, given an iterate $x_k\in\mathcal{X}$, we write
\BEQ\label{eq:hat-delta}
    \Delta_k \triangleq f_{1/(2\rho)}(x_k)-f^*
    \qquad \mbox{and} \qquad
    \hat\Delta_k \triangleq \Delta_k + \frac{L}{\sqrt{\rho}}\Delta_k^{1/2},
\EEQ
so that \Cref{lem:gap-envelope-bound} gives $f(x_k)-f^*\leq\hat\Delta_k$. A complexity result analogous to the following one is shown in \citet[Theorem\,1]{Mai20}.
\begin{lemma}\label{lem:complexity-momentum-auxiliary-point}
    Fix the momentum parameter $\mu \in (0,1]$, a constant $\beta > 0$, and $T > 0$, and set the step size to $\alpha = {\beta}/{\sqrt{T+1}}$. After running \Cref{algo:prox-subgrad-momentum}, the point $\bar x_{t^*} \triangleq x_{t^*} + \frac{1-\mu}{\mu}(x_{t^*}-x_{t^*-1})$ satisfies
    \begin{multline*}
        \Expect[\|\nabla f_{1/(2\rho)}(\bar x_{t^*})\|^2]
        \leq
        \frac{1}{\sqrt{T+1}}\left[ \frac{2\Delta_0}{\beta} + 2\rho L^2\beta \right]
        + \frac{1}{T+1}\left[ \frac{4\rho\hat\Delta_0}{\mu} + \frac{5\rho^2 L^2\beta^2}{\mu} \right] \\
        + \frac{1}{(T+1)^{3/2}}\left[ \frac{4\rho^2\hat\Delta_0\beta}{\mu^2} + \frac{2\rho^3 L^2\beta^3}{\mu^2} \right].
    \end{multline*}
\end{lemma}
\begin{proof}
Under the aforementioned assumptions, \citet[Lem.\,3.2]{Mai20} show that, for $\lambda \leq 1/(2\rho)$, any $t \in \{ 0, \ldots , T-1 \}$, and given $x_0, \ldots, x_t$,
\[
    \Expect[V_{t+1}]
    \leq
    V_t - \frac{\alpha}{2} \Expect[\|\nabla f_\lambda(\bar x_t)\|^2] + \frac{q\alpha^2L^2}{2\lambda},
\]
where
\[
    V_t = f_\lambda(\bar x_t) + \frac{\alpha(1-\mu)^2}{4\lambda^2\mu}\|p_t\|^2 + \frac{\alpha^3(1-\mu)^2}{2\lambda^2\mu^2}\|d_t\|^2 + \left( \frac{\alpha^2(1-\mu)^3}{2\lambda^2\mu^2} + \frac{\alpha(1-\mu)}{\lambda\mu} \right)f(x_{t-1}),
\]
\[
    p_t = \frac{1-\mu}{\mu}(x_t-x_{t-1}), \qquad d_t = \frac{1}{\alpha}(x_{t-1} - x_t)
\]
and $q = 1 + \frac{\rho \alpha (1-\mu)}{2\mu} + \frac{\alpha(1-\mu)^2}{\lambda\mu} + \frac{\rho\alpha^2(1-\mu)^3}{2\lambda\mu^2}$ with the convention that $x_{-1} = x_0$. Taking the expectation, summing over $t=0, \ldots, T$, and fixing $\lambda = 1/(2\rho)$, we obtain
\[
    \Expect[V_{T+1}]
    \leq
    V_0 - \frac{\beta}{2\sqrt{T+1}} \sum_{t=0}^T\Expect[\|\nabla f_{1/(2\rho)}(\bar x_t)\|^2] + q\rho L^2\beta^2.
\]
Setting $q_1 = 2\rho\alpha(1-\mu)/\mu + 2\rho^2\alpha^2(1-\mu)^3/\mu^2$, since $\|p_t\|^2 \geq 0$ and $\|d_t\|^2 \geq 0$, we can lower-bound the left-hand side by $(1+q_1)f^*$. On the other hand, we have
\[
    V_0 = f_\lambda (x_0) + q_1 f(x_0),
\]
since $p_0=d_0=0$. Therefore, noticing that $\frac{1}{T+1}\sum_{t=0}^T \Expect[\|\nabla f_{1/(2\rho)}(\bar x_t)\|^2] = \Expect[\|\nabla f_{1/(2\rho)}(\bar x_{t^*})\|^2]$,
\[
    \frac{\beta\sqrt{T+1}}{2} \Expect[\|\nabla f_{1/(2\rho)}(\bar x_{t^*})\|^2]
    \leq
    f_\lambda (x_0) - f^* + q_1 (f(x_0) - f^*) + q\rho L^2\beta^2.
\]
Hence, bounding $f(x_0)-f^*\leq\hat\Delta_0$ with \Cref{lem:gap-envelope-bound},
\BEQ\label{eq:moreau-gradient-momentum-bound}
    \Expect[\|\nabla f_{1/(2\rho)}(\bar x_{t^*})\|^2]
    \leq
    \frac{2\Delta_0}{\beta\sqrt{T+1}} + \frac{2q_1\hat\Delta_0}{\beta\sqrt{T+1}} + \frac{2q\rho L^2\beta}{\sqrt{T+1}}.
\EEQ
It remains to expand the two constants $q_1$ and $q$. Since $\lambda = 1/(2\rho)$ and $\alpha=\beta/\sqrt{T+1}$, we have
\[
    q_1 = \frac{2\rho\alpha(1-\mu)}{\mu} + \frac{2\rho^2\alpha^2(1-\mu)^3}{\mu^2}
    \leq \frac{2\rho\beta}{\mu\sqrt{T+1}} + \frac{2\rho^2\beta^2}{\mu^2(T+1)}.
\]
Therefore,
\[
    \frac{2q_1\hat\Delta_0}{\beta\sqrt{T+1}} \leq \frac{4\rho\hat\Delta_0}{\mu(T+1)} + \frac{4\rho^2\hat\Delta_0\beta}{\mu^2(T+1)^{3/2}}.
\]
Similarly, using $\lambda=1/(2\rho)$, we obtain
\[
    q = 1 + \frac{\rho\alpha(1-\mu)}{2\mu} + \frac{2\rho\alpha(1-\mu)^2}{\mu} + \frac{\rho^2\alpha^2(1-\mu)^3}{\mu^2}.
\]
Since $1-\mu\leq 1$, this gives
\[
    q \leq 1 + \frac{5\rho\alpha}{2\mu} + \frac{\rho^2\alpha^2}{\mu^2}
    = 1 + \frac{5\rho\beta}{2\mu\sqrt{T+1}} + \frac{\rho^2\beta^2}{\mu^2(T+1)}.
\]
Consequently,
\[
    \frac{2q\rho L^2\beta}{\sqrt{T+1}} \leq \frac{2\rho L^2\beta}{\sqrt{T+1}} + \frac{5\rho^2L^2\beta^2}{\mu(T+1)} + \frac{2\rho^3L^2\beta^3}{\mu^2(T+1)^{3/2}}.
\]
Substituting the two preceding bounds into~\eqref{eq:moreau-gradient-momentum-bound}, we obtain the desired bound.
\end{proof}

The previous result and its proof closely follow \citet[Theorem\,1]{Mai20}. Nevertheless, the bound established here is expressed using the optimality gap $\Delta_0$ of the Moreau envelope, instead of the optimality gap $f(x_0)-f^*$ of the objective function, as is done by \citet{Mai20}. This distinction is why we reproduce the proof here: it allows us to apply the K{\L} inequality on the Moreau envelope when defining the restart scheme for the proximal stochastic subgradient method with momentum. To a certain extent, this result mirrors that of \citet{Davi18} recalled in \Cref{th:complexity}. However, there are two key differences. First, the bound is expressed in terms of the auxiliary point $\bar x_{t^*} = x_{t^*} + \frac{1-\mu}{\mu}(x_{t^*}-x_{t^*-1})$ instead of $x_{t^*}$ itself. Second, we have additional terms of higher order in $1/\sqrt{T}$. These involve the optimality gap of $f$ which we can bound by $\hat\Delta_0$ by using \Cref{lem:gap-envelope-bound}. To adress the first difference, we require the following lemma.
\begin{lemma}\label{lem:complexity-distance-iterate-momentum}
    Fix the momentum parameter $\mu \in (0,1]$, a constant $\beta > 0$, and $T > 0$, and set the step size to $\alpha = {\beta}/{\sqrt{T+1}}$. After running \Cref{algo:prox-subgrad-momentum}, we define
    \[
        d_{t^*} = \frac{1}{\alpha}(x_{t^*-1} - x_{t^*}),
    \]
    then
    \[
        \Expect[\|d_{t^*}\|^2]
        \leq
        \frac{\hat\Delta_0}{\beta \sqrt{T+1}}
        + \frac{\rho L^2 \beta}{2\sqrt{T+1}}
        + \mu L^2,
    \]
    where $\hat\Delta_0$ is defined in~\eqref{eq:hat-delta}.
\end{lemma}

\begin{proof}
Under the aforementioned assumptions, \citet[Lem.\,3.1]{Mai20} show that, for any $t \in \{ 0, \ldots , T-1 \}$ and given $x_0, \ldots, x_t$,
\[
    (1-\mu)f(x_t) + \frac{\mu}{2\alpha}\Expect[\|p_{t+1}\|^2]
    \leq
    (1-\mu)f(x_{t-1}) + \frac{\mu}{2\alpha}\|p_t\|^2 - \alpha \Expect[\|d_{t+1}\|^2] + \alpha^2 \left( \frac{\rho(1-\mu)}{2} + \frac{\mu}{\alpha} \right) L^2,
\]
where $p_t = \frac{1-\mu}{\mu}(x_t -x_{t-1})$. By convention, $x_{-1} = x_0$. Taking the expectation and summing the previous quantities for $t \in \{ 0, \ldots , T-1 \}$ gives
\[
(1-\mu)\Expect[f(x_{T-1})] + \frac{\mu}{2\alpha}\Expect[\|p_T\|^2]
\leq
(1-\mu)f(x_0) - \alpha\sum_{t=0}^T \Expect[\|d_t\|^2] + \alpha^2T \left( \frac{\rho(1-\mu)}{2} + \frac{\mu}{\alpha} \right) L^2,
\]
since $p_0 = d_0 = 0$. Lower-bounding the left-hand side term by $(1-\mu)f^*$ and noticing that $\frac{1}{T+1}\sum_{t=0}^T \Expect[\|d_t\|^2] = \Expect[\|d_{t^*}\|^2]$,
\[
\beta \sqrt{T+1} \Expect[\|d_{t^*}\|^2] \leq (1-\mu)(f(x_0) - f^*) + \frac{\beta^2 T}{T+1} \left( \frac{\rho(1-\mu)}{2} + \frac{\mu\sqrt{T+1}}{\beta} \right) L^2.
\]
Lastly, using $1-\mu \leq 1$, $T < T+1$, and bounding $f(x_0)-f^*\leq\hat\Delta_0$ with \Cref{lem:gap-envelope-bound} gives the desired result.
\end{proof}

Using the two previous results, we can now show the following.

\begin{theorem}\label{th:complexity-momentum}
    Fix the momentum parameter $\mu \in (0,1]$, a constant $\beta > 0$, and $T > 0$, and set the step size to $\alpha = {\beta}/{\sqrt{T+1}}$. After running \Cref{algo:prox-subgrad-momentum}, its output $x_{t^*}$ satisfies
    \begin{multline}\label{eq:complexity-momentum}
        \Expect[\|\nabla f_{1/(2\rho)}(x_{t^*})\|^2]
        \leq
        \frac{1}{\sqrt{T+1}}\left[ \frac{4\Delta_0}{\beta} + 4\rho L^2\beta \right]
        + \frac{1}{T+1}\left[ \frac{8\rho\hat\Delta_0}{\mu} + \frac{18\rho^2 L^2\beta^2}{\mu} \right] \\
        + \frac{1}{(T+1)^{3/2}}\left[ \frac{16\rho^2\hat\Delta_0\beta}{\mu^2} + \frac{8\rho^3 L^2\beta^3}{\mu^2} \right],
    \end{multline}
    where $\Delta_0$ and $\hat\Delta_0$ are defined in~\eqref{eq:hat-delta}.
\end{theorem}

\begin{proof}
Recall that $\bar x_{t^*} \triangleq x_{t^*} + \frac{1-\mu}{\mu}(x_{t^*}-x_{t^*-1})$ and $d_{t^*} = \frac{1}{\alpha}(x_{t^*-1} - x_{t^*})$. Using~\eqref{eq:moreau-smoothness} with $\lambda = (2\rho)^{-1}$, we have
\[
    \| \nabla f_{1/(2\rho)}(x_{t^*}) - \nabla f_{1/(2\rho)}(\bar x_{t^*}) \|
    \leq 2\rho \| x_{t^*} - \bar x_{t^*} \|
    = \frac{2\rho(1-\mu)}{\mu} \| x_{t^*-1} - x_{t^*} \|.
\]
Hence, using the triangle inequality,
\[
\| \nabla f_{1/(2\rho)}(x_{t^*}) \| \leq
\| \nabla f_{1/(2\rho)}(\bar x_{t^*}) \| + \frac{2\alpha\rho}{\mu} \|d_{t^*}\|.
\]
By Young's inequality and taking expectations,
\[
\Expect[\| \nabla f_{1/(2\rho)}(x_{t^*}) \|^2] \leq
2\Expect[\| \nabla f_{1/(2\rho)}(\bar x_{t^*}) \|^2] + \frac{8\rho^2\beta^2}{\mu^2(T+1)} \Expect[\|d_{t^*}\|^2].
\]
\Cref{lem:complexity-momentum-auxiliary-point,lem:complexity-distance-iterate-momentum} allow us to conclude.
\end{proof}

\subsection{Restart scheme}
Having established the preceding preliminary result, we can now define the restarted version of \Cref{algo:prox-subgrad-momentum} and show that it converges at the same rate as its counterpart without momentum. We first show the following result, analogous to \Cref{lem:step}.

\begin{lemma}\label{lem:step-momentum}
    Suppose that $h$ is the indicator of a closed convex set $\mathcal{X}\subseteq\reals^d$ and that $f_{1/(2\rho)}$ satisfies the K{\L} inequality in~\eqref{eq:KL} globally for some $\theta\in[1/2,1]$. Running \Cref{algo:prox-subgrad-momentum} with momentum parameter $\mu \in (0,1]$ and step size
    \[
        \alpha_k = \frac{(f_{1/(2\rho)}(x_k)-f^*)^{1/2}}{\sqrt{\rho} L \sqrt{t_k+1}}
    \]
    for $t_{k} \geq C \exp{((4\theta-1)\gamma k)}$
    iterations at each outer iteration $k$, where
    \begin{multline}\label{eq:tk-constant-momentum}
        C \triangleq \max\left\{ \frac{1}{\mu^2}, \
        \left( 48c^{-2}\sqrt{\rho} L \Delta_0^{1/2-2\theta} e^{2\theta\gamma}\right)^2, \
        126c^{-2}\frac{\rho \Delta_0^{1-2\theta}}{\mu} e^{2\theta\gamma}, \right. \\
        \left. \left(72c^{-2}\frac{\rho^{3/2} \Delta_0^{1-2\theta}}{\mu^2L}\left(\frac{4L}{c}\right)^{1/(2\theta)} e^{2\theta\gamma} \right)^{2/3}  \right\}
    \end{multline}
    ensures $\Expect[f_{1/(2\rho)}(x_{k})]-f^* \leq \Delta_0 \exp(-\gamma k)$.
\end{lemma}

\begin{proof}
We show by induction on $k$ that $\Expect[\Delta_k] \leq \Delta_0 \exp(-\gamma k)$ for every $k\geq0$. For $k=0$, this holds with equality, since $x_0$ is deterministic. Suppose now that it holds at some $k\geq0$, and let $x_{k+1}$ be the output of \Cref{algo:prox-subgrad-momentum} run from $x_k$ for $t_k$ iterations.

Conditioning on $x_k$ and setting $\beta = \Delta_k^{1/2}/(\sqrt{\rho} L)$ in \Cref{th:complexity-momentum}, the three brackets in~\eqref{eq:complexity-momentum} become, respectively,
\[
8\sqrt{\rho}L\Delta_k^{1/2}, \qquad
\frac{26\rho\Delta_k + 8\sqrt{\rho}L\Delta_k^{1/2}}{\mu}
\qquad\mbox{and}\qquad
\frac{24\rho^{3/2}\Delta_k^{3/2}}{\mu^2L} + \frac{16\rho\Delta_k}{\mu^2},
\]
so that
\begin{align*}
\Expect[\|\nabla f_{1/(2\rho)}(x_{k+1})\|^2 \mid x_k] \leq & \frac{8\sqrt{\rho}L\Delta_k^{1/2}}{\sqrt{t_k+1}}
+ \frac{26\rho\Delta_k}{\mu(t_k+1)} + \frac{8\sqrt{\rho}L\Delta_k^{1/2}}{\mu(t_k+1)} \\
& + \frac{24\rho^{3/2}\Delta_k^{3/2}}{\mu^2L(t_k+1)^{3/2}} + \frac{16\rho\Delta_k}{\mu^2(t_k+1)^{3/2}} \\
\leq & \frac{16\sqrt{\rho}L\Delta_k^{1/2}}{\sqrt{t_k+1}} + \frac{42\rho\Delta_k}{\mu(t_k+1)} + \frac{24\rho^{3/2}\Delta_k^{3/2}}{\mu^2L(t_k+1)^{3/2}}.
\end{align*}
The second inequality holds because $\theta\geq1/2$ and $C\geq\mu^{-2}$ give $t_k+1 \geq C\exp((4\theta-1)\gamma k) \geq \mu^{-2}$, i.e., $\mu\sqrt{t_k+1}\geq1$, so that the third term on the right-hand side of the first inequality is bounded by the first one, and the fifth one by the second one.

We now take the total expectation. The first two terms above are, respectively, a concave and a linear function of $\Delta_k$, and are therefore controlled by the induction hypothesis. The third one, on the other hand, is \emph{convex}, and $\Expect[\Delta_k^{3/2}]$ is not bounded by $\left(\Expect[\Delta_k]\right)^{3/2}$. This is the only point where we need a bound on the gap, and \Cref{cor:gap-bounded} provides one: since $x_k\in\mathcal{X}$, we have $\Delta_k\leq(4L/c)^{1/\theta}$, hence
\[
    \Delta_k^{3/2} = \Delta_k^{1/2}\,\Delta_k \leq \left(\frac{4L}{c}\right)^{1/(2\theta)}\Delta_k,
\]
which is linear in $\Delta_k$. After this substitution, the right-hand side is concave in $\Delta_k$, and Jensen's inequality for the concave map $u\mapsto u^{q}$ with $q\in(0,1]$, together with the induction hypothesis, gives
\[
    \Expect[\Delta_k^{q}] \leq \left(\Expect[\Delta_k]\right)^{q} \leq \Delta_0^{q}e^{-q\gamma k},
\]
so that
\[
\Expect[\|\nabla f_{1/(2\rho)}(x_{k+1})\|^2] \leq \frac{16\sqrt{\rho}L\Delta_0^{1/2}e^{-\gamma k/2}}{\sqrt{t_k+1}} + \frac{42\rho\Delta_0 e^{-\gamma k}}{\mu(t_k+1)} + \frac{24\rho^{3/2}\Delta_0 e^{-\gamma k}}{\mu^2L(t_k+1)^{3/2}}\left(\frac{4L}{c}\right)^{1/(2\theta)}.
\]
Taking $t_k \geq C \exp{((4\theta-1)\gamma k)}$ with $C$ as in~\eqref{eq:tk-constant-momentum} ensures that
\begin{align*}
\Expect[\|\nabla f_{1/(2\rho)}(x_{k+1})\|^2] \leq & \frac{16\sqrt{\rho}L\Delta_0^{1/2}}{C^{1/2}}e^{-2\theta\gamma k}
+ \frac{42\rho\Delta_0}{\mu C}e^{-4\theta\gamma k}
+ \frac{24\rho^{3/2}\Delta_0}{\mu^2LC^{3/2}}\left(\frac{4L}{c}\right)^{1/(2\theta)}e^{-(6\theta-1/2)\gamma k} \\
\leq & c^2 \Delta_0^{2\theta}e^{-2\theta\gamma(k+1)} \left( \frac{1}{3} + \frac{e^{-2\theta\gamma k}}{3} + \frac{e^{-(4\theta-1/2)\gamma k}}{3} \right) \\
\leq & c^2 \Delta_0^{2\theta}e^{-2\theta\gamma(k+1)},
\end{align*}
since $\theta \geq 1/2$. As $f_{1/(2\rho)}$ satisfies~\eqref{eq:KL} for some $\theta\in[1/2,1]$, we have, as in~\eqref{eq:SKL},
\[
\Expect[\Delta_{k+1}] \leq c^{-1/\theta} \Expect[\|\nabla f_{1/(2\rho)}(x_{k+1})\|^2]^{1/(2\theta)} \leq \Delta_0e^{-\gamma(k+1)},
\]
which is the induction hypothesis at $k+1$ and completes the proof.
\end{proof}

\Cref{lem:step-momentum} requires no bound on the optimality gap of the objective over $\mathcal{X}$, unlike the analysis of \citet{Mai20}. We now define the restarted proximal stochastic subgradient method with momentum.

\begin{algorithm}[H]
	\caption{Restarted proximal subgradient with momentum}\label{algo:prox-subgrad-momentum-restart}
    \begin{algorithmic}
        \State{\textbf{Inputs:} $x_0\in\reals^d,~\mu\in(0,1],~\gamma>0,~N\in\mathbb{N}$.}
        \For{$k=0,\ldots,N-1 $}
        \State Run \Cref{algo:prox-subgrad-momentum}, starting at $x_{k}$, with momentum parameter $\mu$ for $t_{k} = C\exp((4\theta-1)\gamma k)$ iterations, where $C$ is defined as in~\eqref{eq:tk-constant-momentum}, with step size
        \[
            \alpha_k = \frac{(f_{1/(2\rho)}(x_k)-f^*)^{1/2}}{\sqrt{\rho} L \sqrt{t_k+1}}
        \]
        and set $x_{k+1}$ as its output.
        \EndFor
        \State{\textbf{Output:} $x_{N}$}
    \end{algorithmic}
\end{algorithm}

\begin{theorem}\label{th:primal-gap-complexity-bound-momentum}
    Suppose $h$ is the indicator of a closed convex set $\mathcal{X}\subseteq\reals^d$ and that $f_{1/(2\rho)}$ satisfies the K{\L} inequality in~\eqref{eq:KL} globally for some $\theta\in[1/2,1]$. Let $T \triangleq \sum_{k=0}^{N-1} t_k$ be the total inner-iteration budget after $N$ outer iterations of the restart scheme in \Cref{algo:prox-subgrad-momentum-restart}. Then
    \BEQ\label{eq:restart-complexity-momentum}
        \Expect[f_{1/(2\rho)}(x_N)]-f^* \leq \frac{f_{1/(2\rho)}(x_0)-f^*}{\left(\gamma(4\theta-1)C^{-1} T + 1\right)^{1/(4 \theta -1)}}
    \EEQ
    where $C$ is defined as in~\eqref{eq:tk-constant-momentum}.
\end{theorem}
\begin{proof}
    \Cref{lem:step-momentum} shows that setting $t_{k} \geq C \exp{((4\theta-1)\gamma k)}$ ensures $\Expect[f_{1/(2\rho)}(x_{k})]-f^* \leq \Delta_0 \exp(-\gamma k)$. We again conclude using \citet[Lem.\,2.1]{Roul17}.
\end{proof}

\Cref{th:primal-gap-complexity-bound-momentum} shows that the restart scheme retains the same convergence rate when momentum is added.

\section{Acceleration without restarts}
\label{sec:accel-without-rest}

In this section, we show that restart schemes are not required to obtain the improved rates detailed above, provided that the step sizes are adapted to the K{\L} exponent.

\begin{theorem}\label{th:accel-no-restart}
    Suppose $h$ is the indicator of a closed convex set $\mathcal{X}\subseteq\reals^d$, that $f_{1/(2\rho)}$ satisfies the K{\L} inequality \labelcref{eq:KL} globally for some $\theta\in[1/2,1]$, and that $\Delta_0 \triangleq f_{1/(2\rho)}(x_0)-f^*$ satisfies $c^2\Delta_0^{2\theta}\le 4L^2$. Running \Cref{algo:prox-subgrad} for $T$ iterations, returning the final iterate $x_T$ rather than a sample $x_{t^*}$, with step size
    \BEQ\label{eq:accel-step}
        \alpha_t = \frac{c^2(\Expect[f_{1/(2\rho)}(x_t)-f^*])^{2\theta}}{8\rho L^2}
    \EEQ
    at each iteration $t$ ensures
    \BEQ\label{eq:accel-rate}
        \Expect[f_{1/(2\rho)}(x_T)]-f^* \le \left(\mu(4\theta-1)T\right)^{-1/(4\theta-1)},
    \EEQ
    where $\mu \triangleq c^4/(32\rho L^2)$.
\end{theorem}
\begin{proof}
By \citet[Theorem\,3.4]{Davi19}, we have
\begin{equation}
  \label{eq:davis-prox}
 \Expect [f_{1/\bar{\rho}}(x_{t+1})] \le  \Expect [f_{1/\bar{\rho}}(x_{t})]  -\frac{\alpha_t(\bar\rho-\rho)}{\bar\rho} \Expect[\norm{\nabla f_{1/\bar\rho}(x_t)}^2]+\alpha_t^2\bar\rho L^2
\end{equation}
for $0<\alpha_t\le 1/\bar{\rho}$ and $\rho<\bar{\rho}\le 2\rho$.

Setting $\bar{\rho}= 2\rho$ and applying~\eqref{eq:kl-moreau}, we get
\begin{align*}
 \Expect [f_{1/(2\rho)}(x_{t+1})] &\le  \Expect [f_{1/(2\rho)}(x_{t})]  -\frac{\alpha_t}{2} \Expect[\norm{\nabla f_{1/(2\rho)}(x_t)}^2]+2\alpha_t^2\rho L^2\\
  &\le  \Expect [f_{1/(2\rho)}(x_{t})]  -\frac{\alpha_t}{2} c^2(\Expect [f_{1/(2\rho)}(x_{t})-f^*])^{2\theta}+2\alpha_t^2\rho L^2.
\end{align*}
Minimizing the right-hand side over $\alpha_t$ yields the step size~\eqref{eq:accel-step}. For \citet[Theorem\,3.4]{Davi19} to apply to the proximal variant of SGD (as opposed to projected SGD) at iteration $t$, we additionally need $\alpha_t\le 1/(2\rho)$, i.e., $c^2(\Expect[f_{1/(2\rho)}(x_t)-f^*])^{2\theta}\le 4L^2$, which holds at $t=0$ by assumption. Provided that it holds at iteration $t$, plugging in $\alpha_t$ and subtracting $f^*$ from both sides gives
\begin{equation}
 \label{eq:accel-descent}
 \Expect [f_{1/(2\rho)}(x_{t+1})-f^*] \le  \Expect [f_{1/(2\rho)}(x_{t})-f^*]  -\frac{c^4(\Expect [f_{1/(2\rho)}(x_{t})-f^*])^{4\theta}}{32\rho L^2}.
\end{equation}
Since the subtracted term is nonnegative, \eqref{eq:accel-descent} shows that $\Expect[f_{1/(2\rho)}(x_{t+1})-f^*]\le\Expect[f_{1/(2\rho)}(x_t)-f^*]$ whenever it applies. A simple induction then shows that $\Expect[f_{1/(2\rho)}(x_t)-f^*]\le\Delta_0$ for every $t\ge0$, so that the step size constraint indeed holds at every iteration. The base case $t=0$ is trivial. If $\Expect[f_{1/(2\rho)}(x_t)-f^*]\le\Delta_0$, then $c^2(\Expect[f_{1/(2\rho)}(x_t)-f^*])^{2\theta}\le c^2\Delta_0^{2\theta}\le 4L^2$, so $\alpha_t\le1/(2\rho)$, \eqref{eq:accel-descent} applies at iteration $t$, and $\Expect[f_{1/(2\rho)}(x_{t+1})-f^*]\le\Expect[f_{1/(2\rho)}(x_t)-f^*]\le\Delta_0$, which closes the induction. Hence~\eqref{eq:accel-descent} holds unconditionally for every $t\ge0$.
Writing $\Delta_t\triangleq\Expect [f_{1/(2\rho)}(x_{t})-f^*]$ and $\mu\triangleq c^4/(32\rho L^2)$, we have
\begin{equation*}
  \Delta_{t+1}\le \Delta_t-\mu\Delta_t^{4\theta}= \Delta_t(1-\mu\Delta_t^{4\theta-1}).
\end{equation*}
By definition, $\Delta_{t}\ge 0$ for all $t$, so the above inequality implies $\mu\Delta_t^{4\theta-1}\le 1$ whenever $\Delta_t>0$. If $\Delta_{t+1}=0$ for some $t<T$, then $\Delta_T=0$ since $(\Delta_t)$ is nonincreasing, and the result holds trivially. We may thus assume without loss of generality that $\Delta_{t+1}>0$, hence $\Delta_t>0$ and $1-\mu\Delta_t^{4\theta-1}<1$.
Since $\Delta_t,\Delta_{t+1}>0$ and $4\theta-1\in[1,3]$, we can raise both sides of the inequality to the power $-(4\theta-1)$, which gives
\begin{equation*}
 \Delta_t^{-(4\theta-1)}(1-\mu\Delta_t^{4\theta-1})^{-(4\theta-1)} \le \Delta_{t+1}^{-(4\theta-1)}.
\end{equation*}
Applying the inequality $(1-z)^{-r}\ge 1+rz$, valid for $z\in [0,1)$ and $r>0$, with $z= \mu\Delta_t^{4\theta-1}$ and $r= 4\theta-1$ gives
\begin{align*}
\Delta_t^{-(4\theta-1)}+(4\theta-1)\mu&\le \Delta_t^{-(4\theta-1)}(1+(4\theta-1)\mu\Delta_t^{4\theta-1})\le \Delta_{t+1}^{-(4\theta-1)}\\
(4\theta-1)\mu&\le \Delta_{t+1}^{-(4\theta-1)}-\Delta_t^{-(4\theta-1)}
\end{align*}
Summing from $t=0$ to $T-1$, we get
\begin{equation*}
T(4\theta-1)\mu\le \Delta_{T}^{-(4\theta-1)}-\Delta_0^{-(4\theta-1)}.
\end{equation*}
Dropping the negative term and rearranging, we have
\[
  \Delta_{T}\le \left(\mu(4\theta-1)T\right)^{-1/(4\theta-1)},
\]
which is~\eqref{eq:accel-rate}.
\end{proof}

The assumption $c^2\Delta_0^{2\theta}\le 4L^2$ in \Cref{th:accel-no-restart} can always be satisfied by taking $L$ large enough, since $L^2$ is only an upper bound on $\Expect\!_\xi[\|G(x,\xi)\|^2]$. It also admits a simple sufficient condition in terms of the distance from $x_0$ to a minimizer $x^*$ of $f$, with no dependence on the decomposition $f=g+h$.
Indeed, the $2\rho$-Lipschitz continuity of $\nabla f_{1/(2\rho)}$ in~\eqref{eq:moreau-smoothness} (with $\lambda=(2\rho)^{-1}$), combined with the smoothness of $\nabla f_{1/(2\rho)}$ at $x^*$, gives
\[
    \Delta_0 = f_{1/(2\rho)}(x_0)-f^* \le \big\langle \nabla f_{1/(2\rho)}(x^*), x_0-x^*\big\rangle + \rho\|x_0-x^*\|^2 = \rho\|x_0-x^*\|^2.
\]
Hence $c^2\Delta_0^{2\theta}\le c^2\rho^{2\theta}\|x_0-x^*\|^{4\theta}$, so it suffices that
\[
    c\,\rho^\theta\,\|x_0-x^*\|^{2\theta} \le 2L, \qquad \mbox{i.e.} \qquad \|x_0-x^*\| \le \left(\frac{2L}{c\rho^\theta}\right)^{1/(2\theta)}.
\]

\Cref{th:accel-no-restart} thus achieves the same rate as the restart scheme in \Cref{th:primal-gap-complexity-bound}. However, its step size~\eqref{eq:accel-step} depends on the expected optimality gap, which is not observable in practice. The following result replaces this quantity with a deterministic sequence.

\begin{theorem}\label{th:accel-no-restart-deterministic}
    Suppose $f_{1/(2\rho)}$ satisfies the K{\L} inequality \labelcref{eq:KL} globally for some $\theta\in[1/2,1]$, and that $\Delta_0 \triangleq f_{1/(2\rho)}(x_0)-f^*$ satisfies $c^2\Delta_0^{2\theta}\le 4L^2$, as well as $4\theta\mu\Delta_0^{4\theta-1} \leq 1$, where $\mu \triangleq c^4/(32\rho L^2)$. Running \Cref{algo:prox-subgrad} for $T$ iterations, returning the final iterate $x_T$ rather than a sample $x_{t^*}$, with step size
    \BEQ\label{eq:accel-step-deterministic}
        \alpha_t = \frac{c^2D_t^{2\theta}}{8\rho L^2}
    \EEQ
    at each iteration $t$, where $D_t = (\mu(4\theta -1)t + \Delta_0^{-(4\theta -1)})^{-1/(4\theta-1)}$, ensures
    \BEQ\label{eq:accel-rate-deterministic}
        \Expect[f_{1/(2\rho)}(x_T)]-f^* \le (\mu(4\theta -1)T + \Delta_0^{-(4\theta -1)})^{-1/(4\theta-1)}.
    \EEQ
\end{theorem}
\begin{proof}
For $t \geq 0$, let $\Delta_t \triangleq \Expect[f_{1/(2\rho)}(x_t)]-f^*$. Since $D_t$ is decreasing and $D_0 = \Delta_0$, we have $c^2D_t^{2\theta} \leq c^2\Delta_0^{2\theta} \leq 4L^2$, hence $\alpha_t\le 1/(2\rho)$. As in the proof of \Cref{th:accel-no-restart}, \citet[Theorem\,3.4]{Davi19} gives
\[
    \Expect [f_{1/(2\rho)}(x_{t+1})] \le  \Expect [f_{1/(2\rho)}(x_{t})]  -\frac{\alpha_t}{2} \Expect[\norm{\nabla f_{1/(2\rho)}(x_t)}^2]+2\alpha_t^2\rho L^2.
\]
Applying~\eqref{eq:kl-moreau} and subtracting $f^*$ from both sides gives
\[
    \Delta_{t+1} \le \Delta_t - \frac{\alpha_t}{2} c^2\Delta_t^{2\theta}+2\alpha_t^2\rho L^2.
\]
Substituting the value of $\alpha_t$ gives
\[
    \Delta_{t+1} \le \Delta_t - 2\mu D_t^{2\theta}\Delta_t^{2\theta} + \mu D_t^{4\theta}.
\]
Now, assume by induction that $\Delta_t \leq D_t$ (which holds at $t=0$). For $a \in [0,D_t]$, consider $g(a) \triangleq a - 2\mu D_t^{2\theta}a^{2\theta} + \mu D_t^{4\theta}$. We have
\[
    g'(a) = 1 - 4\theta\mu D_t^{2\theta}a^{2\theta-1} \geq 1 - 4\theta\mu D_t^{4\theta-1} \geq 1 - 4\theta\mu \Delta_0^{4\theta-1} \geq 0,
\]
where we used $2\theta - 1 \geq 0$, $D_t \leq \Delta_0$, as well as $4\theta\mu\Delta_0^{4\theta-1} \leq 1$. Hence $g$ is nondecreasing; therefore,
\[
    \Delta_{t+1} \le D_t - \mu D_t^{4\theta} = D_t(1-\mu D_t^{4\theta-1}).
\]
As in the previous proof, we only consider the case where $\Delta_{t+1} > 0$. Furthermore, $\mu D_t^{4\theta-1} \leq \mu \Delta_0^{4\theta-1} \leq 1$, and since $(1-z)^{-r} \geq 1+rz$ for $z\in [0,1)$, we have
\begin{align*}
    \Delta_{t+1}^{-(4\theta-1)} & \geq D_t^{-(4\theta-1)}(1-\mu D_t^{4\theta-1})^{-(4\theta-1)} \\
    & \geq D_t^{-(4\theta-1)}(1+\mu (4\theta-1) D_t^{4\theta-1}) \\
    & \geq D_t^{-(4\theta-1)} + \mu (4\theta-1) \\
    & \geq \mu (4\theta-1) (t+1) + \Delta_0^{-(4\theta -1)} = D_{t+1}^{-(4\theta-1)}.
\end{align*}
Hence $\Delta_{t+1} \leq D_{t+1}$, which concludes the induction and proves that $\Delta_T \leq D_T$.
\end{proof}

Note that the discussion regarding the assumptions $c^2\Delta_0^{2\theta}\le 4L^2$ after \cref{th:accel-no-restart} also applies to this result.

\section{Adaptation}\label{sec:adapt}

In this section, we show that the complexity of restart schemes is robust to misspecification of the restart schedule. Since the restart schedule depends on the K{\L} exponent, which is usually unknown, this means that restart schemes are nearly adaptive to this exponent.

\subsection{Constant restart scheme} We now consider the case where we restart \Cref{algo:prox-subgrad} with a constant number of inner iterations $t$ at each outer iteration, instead of the exponential restart schedule detailed in \Cref{sec:restart}, and with step sizes defined as in~\eqref{eq:lr}.

\begin{algorithm}[H]
	\caption{Restarted proximal subgradient with constant restart schedule}\label{algo:prox-subgrad-restart-constant-step}
		\begin{algorithmic}
			\State{\textbf{Inputs:} $x_0\in\reals^d,~N\in\mathbb{N},~t\in\mathbb{N}^*$.}
			\For{$k=0,\ldots,N-1 $}
			\State Run \Cref{algo:prox-subgrad}, starting at $x_{k}$, for $t$ iterations, with step size
	        \[
	            \alpha_k = \frac{(f_{1/(2\rho)}(x_k)-f^*)^{1/2}}{\sqrt{\rho} L \sqrt{t+1}}
	        \]
            and set $x_{k+1}$ as its output.
			\EndFor
	        \State{\textbf{Output:} $x_{N}$}
		\end{algorithmic}
\end{algorithm}

We show the following result.

\begin{lemma}\label{lem:constant-step}
    Suppose $h$ is the indicator of a closed convex set $\mathcal{X}\subseteq\reals^d$ and $f_{1/(2\rho)}$ satisfies the K{\L} inequality in~\eqref{eq:KL} globally for some $\theta\in[1/2,1]$. Running \Cref{algo:prox-subgrad} with step size
    \[
        \alpha_k = \frac{(f_{1/(2\rho)}(x_k)-f^*)^{1/2}}{\sqrt{\rho} L \sqrt{t+1}}
    \]
    for $t\geq 0$ inner iterations at each outer iteration $k$ ensures
    \BEQ\label{eq:restart-complexity-constant-step-lemma}
    \Expect[f_{1/(2\rho)}(x_{k})]-f^* \leq \Delta_0 \left( \frac{16c^{-4} \Delta_0^{1-4\theta}\rho L^2}{t} \right)^{\frac{1-(4\theta)^{-k}}{4\theta-1}},
    \EEQ
    where $\Delta_0 = f_{1/(2\rho)}(x_0)-f^*$.
\end{lemma}

\begin{proof}
    According to \Cref{th:complexity}, taking 
    \[
        \beta = \frac{(f_{1/(2\rho)}(x_k)-f^*)^{1/2}}{\sqrt{\rho}L} ,
    \]
    we have
    \[
        \Expect[ \| \nabla f_{1/(2\rho)} (x_{k+1}) \|^2 ]
        \leq \frac{4\sqrt{\rho}L(f_{1/(2\rho)}(x_k)-f^*)^{1/2}}{\sqrt{t + 1}}.
    \]
     Taking the expectation, we can use Jensen's inequality since $x \mapsto x^{1/2}$ is concave. This gives
    \[
        \Expect[\|\nabla f_{1/(2\rho)}(x_{k+1})\|^2]
        \leq \Expect\!\left[\frac{4\sqrt{\rho}L(f_{1/(2\rho)}(x_k)-f^*)^{1/2}}{\sqrt{t+1}}\right]
        \leq \frac{4\sqrt{\rho}L(\Expect[f_{1/(2\rho)}(x_{k})]-f^*)^{1/2}}{\sqrt{t+1}}.
    \]
    Assuming $f_{1/(2\rho)}$ satisfies~\eqref{eq:KL}, by~\eqref{eq:SKL} we have
    \[
        \Expect[\|\nabla f_{1/(2\rho)}(x_{k+1})\|^2]
        \geq c^2(\Expect[f_{1/(2\rho)}(x_{k+1})]-f^*)^{2\theta}
    \]
    for some $\theta \in [1/2,1]$.
    Hence, by setting $\Delta_k = \Expect[f_{1/(2\rho)}(x_{k})] - f^*$, we obtain
    \[
    c^2 \Delta_{k+1}^{2\theta} \leq \frac{4\sqrt{\rho}L\Delta_k^{1/2}}{\sqrt{t + 1}},
    \]
    therefore, 
    \begin{align*}
        \Delta_k & \leq \left(\frac{16c^{-4}\rho L^2}{t+1}\right)^{1/(4\theta)} \Delta_{k-1}^{1/(4\theta)} \leq \left(\frac{16c^{-4}\rho L^2}{t+1}\right)^{1/(4\theta)} \left(\frac{16c^{-4}\rho L^2}{t+1}\right)^{1/(4\theta)^2} \Delta_{k-2}^{1/(4\theta)^2} \\ & \leq \ldots \leq \left(\frac{16c^{-4}\rho L^2}{t+1}\right)^{\sum_{i=1}^k1/(4\theta)^i} \Delta_0^{1/(4\theta)^k} .
    \end{align*} 
    Hence,
    \[ 
    \Delta_k
    \leq \left(\frac{16c^{-4}\rho L^2}{t+1}\right)^{\frac{1-(4\theta)^{-k}}{4\theta-1}}\Delta_0^{(4\theta)^{-k}}
    \leq \Delta_0 \left( \frac{16c^{-4} \Delta_0^{1-4\theta}\rho L^2}{t} \right)^{\frac{1-(4\theta)^{-k}}{4\theta-1}},
    \] 
    which is the desired result.
\end{proof}

A similar result holds for the gradient norm of the Moreau envelope.

\begin{corollary}\label{cor:constant-step-moreau-gradient}
    Suppose $h$ is the indicator of a closed convex set $\mathcal{X}\subseteq\reals^d$ and $f_{1/(2\rho)}$ satisfies the K{\L} inequality in~\eqref{eq:KL} globally for some $\theta\in[1/2,1]$. Running \Cref{algo:prox-subgrad} with step size
    \[
        \alpha_k = \frac{(f_{1/(2\rho)}(x_{k})-f^*)^{1/2}}{\sqrt{\rho} L \sqrt{t+1}}
    \]
    for $t\geq 0$ inner iterations at each outer iteration $k$ ensures
    \BEQ
    \Expect[ \| \nabla f_{1/(2\rho)}(x_k) \|^2 ]^{1/2} \leq \|\nabla f_{1/(2\rho)}(x_0)\| \left( \frac{16c^{-4}\Delta_0^{1-4\theta}\rho L^2}{t} \right)^{\frac{\theta(1-(4\theta)^{-k})}{4\theta-1}},
    \EEQ
    where $\Delta_0 = f_{1/(2\rho)}(x_0)-f^*$.
\end{corollary}

\begin{proof}    
    We simply have
    \begin{align}
        \Expect[ \|\nabla f_{1/(2\rho)}(x_k)\|^2 ]^{1/2} & \leq \left( \frac{16\rho L^2}{t} \right)^{1/4}(\Expect[f_{1/(2\rho)}(x_{k-1})]-f^*)^{1/4} \label{eq:moreau-gradient-bound-1} \\
        & \leq \Delta_0^{1/4} \left( \frac{16\rho L^2}{t} \right)^{1/4} \left( \frac{16c^{-4}\Delta_0^{1-4\theta}\rho L^2}{t} \right)^{\frac14 \cdot \frac{1-(4\theta)^{-k+1}}{4\theta-1}} \label{eq:moreau-gradient-bound-2} \\
        & \leq c\Delta_0^{\theta}\left( \frac{16c^{-4}\Delta_0^{1-4\theta}\rho L^2}{t} \right)^{\frac14 \cdot \frac{4\theta-(4\theta)^{-k+1}}{4\theta-1}} \nonumber \\
        & \leq \|\nabla f_{1/(2\rho)}(x_0)\|\left( \frac{16c^{-4}\Delta_0^{1-4\theta}\rho L^2}{t} \right)^{\theta \cdot \frac{1-(4\theta)^{-k}}{4\theta-1}}, \label{eq:moreau-gradient-bound-4}
    \end{align}
    where in~\eqref{eq:moreau-gradient-bound-1}, we used \Cref{th:complexity}. For~\eqref{eq:moreau-gradient-bound-2}, we used~\eqref{eq:restart-complexity-constant-step-lemma}. Finally, for~\eqref{eq:moreau-gradient-bound-4}, we used the K{\L} inequality~\eqref{eq:KL} at $x_0$. This gives the desired result.
\end{proof}

We can now use \Cref{lem:constant-step} to show a bound on the precision reached after a total budget of $T$ iterations, as a function of the number of restarts $N$.

\begin{theorem}\label{th:constant-restart}
    Suppose $h$ is the indicator of a closed convex set $\mathcal{X}\subseteq\reals^d$ and $f_{1/(2\rho)}$ satisfies the K{\L} inequality in~\eqref{eq:KL} globally for some $\theta\in[1/2,1]$. After $N$ outer iterations of the restart scheme in \Cref{algo:prox-subgrad-restart-constant-step}, define the total inner-iteration budget $T \triangleq Nt$.
    Then
    \BEQ\label{eq:restart-complexity-constant-step}
        \Expect[f_{1/(2\rho)}(x_N)]-f^* \leq \Delta_0 \left(\frac{CN}{T}\right)^{\frac{1-(4\theta)^{-N}}{4\theta-1}},
    \EEQ
    and
    \BEQ\label{eq:moreau-gradient-restart-complexity-constant-step}
        \Expect[ \| \nabla f_{1/(2\rho)}(x_N) \|^2 ]^{1/2} \leq \|\nabla f_{1/(2\rho)}(x_0)\| \left( \frac{CN}{T} \right)^{\frac{\theta(1-(4\theta)^{-N})}{4\theta-1}},
    \EEQ
    where $\Delta_0 = f_{1/(2\rho)}(x_0)-f^*$ and $C = 16c^{-4}\Delta_0^{1-4\theta}\rho L^2$.
\end{theorem}

\begin{proof}
    The result follows directly from \Cref{lem:constant-step,cor:constant-step-moreau-gradient} at outer iteration $N$, with $t = T/N$.
\end{proof}

For this constant restart scheme, we can show that there is an optimal number of restarts $N^*$ for a fixed budget $T$.

\begin{lemma}
    For a fixed total budget $T$, there exists a unique $N^*$ such that 
    \[ 
    N^* = \argmin_{x\in\mathbb{R}_+} \left\{\psi(x) \triangleq \Delta_0 \left(\frac{Cx}{T}\right)^{\frac{1-(4\theta)^{-x}}{4\theta-1}}\right\}.
    \]
    Furthermore, $N^*$ satisfies
    \BEQ\label{eq:optimality-condition-constant-step}
    N^* = \frac{T}{C}\exp{\frac{1-(4\theta)^{N^*}}{N^*\log{(4\theta)}}},
    \EEQ
    where $\Delta_0 = f_{1/(2\rho)}(x_0)-f^*$ and $C = 16c^{-4}\Delta_0^{1-4\theta}\rho L^2$.
\end{lemma}

\begin{proof}
    We proceed by studying the following function for $x \in \mathbb{R}_+$
    \[ 
    \psi(x) = \Delta_0 \left(\frac{Cx}{T}\right)^{\frac{1-(4\theta)^{-x}}{4\theta-1}}. 
    \]
    We have
    \[ 
    \psi'(x) = \left[ \frac{1-(4\theta)^{-x}}{x(4\theta-1)} + \frac{(4\theta)^{-x}}{4\theta-1}\log{(4\theta)}\log{\frac{Cx}{T}} \right] \psi(x) .
    \]
    Hence, 
    \begin{align*}
        \psi'(x) = 0 & \Leftrightarrow \frac{1-(4\theta)^{-x}}{x(4\theta-1)} + \frac{(4\theta)^{-x}}{4\theta-1}\log{(4\theta)}\log{\frac{Cx}{T}} = 0 \\
        & \Leftrightarrow \log{\frac{Cx}{T}} = \frac{1-(4\theta)^{x}}{x\log{(4\theta)}} \\
        & \Leftrightarrow x = \frac{T}{C}\exp{\frac{1-(4\theta)^{x}}{x\log{(4\theta)}}}
    \end{align*}
    We can express $\psi'$ as follows
    \[ \psi'(x) = \frac{\log{(4\theta)}}{(4\theta)^x(4\theta-1)} \left[ \frac{(4\theta)^x-1}{x\log{(4\theta)}} + \log{\frac{Cx}{T}} \right] \psi(x) .\]
    It is easy to see that
    \[ \lim_{x\rightarrow 0^+} \psi'(x) = -\infty , \]
    and that $\psi'(x) > 0$ for $x$ sufficiently large. Specifically, $\psi'(T/C) > 0$. By continuity of $\psi'$ and the intermediate value theorem, there exists $N^* \in ( 0, T/C )$ such that $\psi'(N^*) = 0$. This $N^*$ is unique. Indeed, for $x > 0$, we have $\frac{\log{(4\theta)}}{(4\theta)^x(4\theta-1)} > 0$ and $\psi(x) > 0$; therefore, the sign of $\psi'(x)$ is that of the term $\frac{(4\theta)^x-1}{x\log{(4\theta)}} + \log{\frac{Cx}{T}}$, which is increasing. Using the previous result, we have 
    \[ 
    N^* = \frac{T}{C}\exp{\frac{1-(4\theta)^{N^*}}{N^*\log{(4\theta)}}},
    \]
    which is the desired result.
\end{proof}

Here, we chose to minimize the bound on the optimality gap in~\eqref{eq:restart-complexity-constant-step}. We would obtain the same minimizer $N^*$ by minimizing the bound in~\eqref{eq:moreau-gradient-restart-complexity-constant-step} instead. We now show the following result on $N^*$ as a function of $T$.

\begin{theorem}
    We fix a total budget $T \in \mathbb{N}^*$. Defining 
    \BEQ\label{eq:Nstar-formulation}
    N^* = \argmin_{x\in\mathbb{R}_+} \Delta_0 \left(\frac{Cx}{T}\right)^{\frac{1-(4\theta)^{-x}}{4\theta-1}},
    \EEQ
    we have
    \BEQ\label{eq:Nstar-equivalence}
        N^* \underset{T\rightarrow+\infty}{\sim} \frac{\log{\log{(T/C)}}}{\log{(4\theta)}},
    \EEQ
    hence after $N^*$ outer iterations of the restart scheme in \Cref{algo:prox-subgrad-restart-constant-step},
    \BEQ\label{eq:equivalent-optimality-bound}
        \Expect[f_{1/(2\rho)}(x_{N^*})]-f^* \underset{T\rightarrow+\infty}{\lesssim} \Delta_0 \left( \frac{C\log{\log{(T/C)}}}{T \log{(4\theta)}} \right)^{1/(4\theta-1)},
    \EEQ
    and
    \BEQ\label{eq:moreau-gradient-equivalent-optimality-bound}
        \Expect[ \| \nabla f_{1/(2\rho)}(x_{N^*}) \|^2 ]^{1/2} \underset{T\rightarrow+\infty}{\lesssim}  \|\nabla f_{1/(2\rho)}(x_0)\| \left( \frac{C\log{\log{(T/C)}}}{T \log{(4\theta)}} \right)^{\theta/(4\theta-1)},
    \EEQ
    where $\Delta_0 = f_{1/(2\rho)}(x_0)-f^*$ and $C = 16c^{-4}\Delta_0^{1-4\theta}\rho L^2$.
\end{theorem}

\begin{proof}
    We consider $p > 1$, as well as the function
    \[
        \Psi(x) = \frac{(4\theta)^x-1}{x\log{(4\theta)}} + \log{\frac{Cx}{T}}
    \]
    and look at its sign on $\left[ \frac{\log{\log{(T/C)}}}{\log{(4\theta)}} , \frac{p\log{\log{(T/C)}}}{\log{(4\theta)}} \right] $.
    \begin{align*}
    \Psi\left(\frac{\log{\log{(T/C)}}}{\log{(4\theta)}}\right) &= \frac{\log{(T/C)} -1}{\log{\log{(T/C)}}} + \log{\frac{C\log{\log{(T/C)}}}{T \log{(4\theta)}}} \\
    &= \frac{\log{(T/C)}-1}{\log{\log{(T/C)}}} + \log{\log{\log{(T/C)}}} - \log{(T/C)} - \log{\log{(4\theta)}} .
    \end{align*}
    We have
    \[ 
        \frac{\log{(T/C)}-1}{\log{\log{(T/C)}}} \cdot \frac{1}{\log{(T/C)}} \underset{T\rightarrow+\infty}{\sim} \frac{1}{\log{\log{(T/C)}}} \xrightarrow[T\rightarrow+\infty]{}0
    \]
    therefore, $\frac{\log{(T/C)}-1}{\log{\log{(T/C)}}} = o_{T\rightarrow+\infty}\left(\log{(T/C)}\right)$. Similarly, $\log{\log{\log{(T/C)}}} = o_{T\rightarrow+\infty}\left(\log{(T/C)}\right)$.
    Hence, for $T$ sufficiently large, the negative term $-\log{(T/C)}$ dominates and
    \[
        \Psi\left(\frac{\log{\log{(T/C)}}}{\log{(4\theta)}}\right)<0.
    \]
    Likewise,
    \begin{align*}
        \Psi\left(\frac{p\log{\log{(T/C)}}}{\log{(4\theta)}}\right) &= \frac{\left(\log{(T/C)}\right)^p-1}{p\log{\log{(T/C)}}} + \log{\frac{Cp\log{\log{(T/C)}}}{T \log{(4\theta)}}} \\
        &= \frac{\left(\log{(T/C)}\right)^p-1}{p\log{\log{(T/C)}}} + \log{\log{\log{(T/C)}}} - \log{(T/C)} - \log{\frac{\log{(4\theta)}}{p}} .
    \end{align*}
    Again, we can see that
    \[ \frac{\left(\log{(T/C)}\right)^p-1}{p\log{\log{(T/C)}}} \cdot \frac{1}{\log{(T/C)}} \underset{T\rightarrow+\infty}{\sim} \frac{\left(\log{(T/C)}\right)^{p-1}}{p\log{\log{(T/C)}}} \xrightarrow[T\rightarrow+\infty]{}+\infty . \]
    For $T$ sufficiently large,
    \[
    \Psi\left(\frac{p\log{\log{(T/C)}}}{\log{(4\theta)}}\right)>0.
    \]
    Thus, for $T$ sufficiently large, $N^* \in \left[ \frac{\log{\log{(T/C)}}}{\log{(4\theta)}} , \frac{p\log{\log{(T/C)}}}{\log{(4\theta)}} \right] $. Since $p > 1$ is arbitrary, this shows that
    \[ N^* \underset{T\rightarrow+\infty}{\sim} \frac{\log{\log{(T/C)}}}{\log{(4\theta)}} .\]
    Using~\eqref{eq:restart-complexity-constant-step}, we have
    \[
        \Expect[f_{1/(2\rho)}(x_{N^*})]-f^* \leq \Delta_0 \left(\frac{CN^*}{T}\right)^{\frac{1-(4\theta)^{-N^*}}{4\theta-1}}.
    \]
    Using the optimality condition~\eqref{eq:optimality-condition-constant-step} and the fact that $N^* \xrightarrow[T\rightarrow +\infty]{} +\infty$, we get
    \[
        (4\theta)^{-N^*}\log{\frac{CN^*}{T}} = \frac{(4\theta)^{-N^*}-1}{N^*\log{(4\theta)}} \xrightarrow[T\rightarrow +\infty]{} 0,
    \]
    hence
    \[
        \log{\biggl[\left(\frac{CN^*}{T}\right)^{\frac{1-(4\theta)^{-N^*}}{4\theta-1}}\biggr]} = \frac{1-(4\theta)^{-N^*}}{4\theta-1} \log{\frac{CN^*}{T}} = \frac{1}{4\theta-1} \log{\frac{CN^*}{T}} + o_{T\rightarrow+\infty}(1),
    \]
    therefore,
    \begin{multline*}
        \left(\frac{CN^*}{T}\right)^{\frac{1-(4\theta)^{-N^*}}{4\theta-1}}
        = \left(\frac{CN^*}{T}\right)^{1/(4\theta-1)}e^{o(1)} \\
        \underset{T\rightarrow+\infty}{\sim} \left(\frac{CN^*}{T}\right)^{1/(4\theta-1)}
        \underset{T\rightarrow+\infty}{\sim} \left(\frac{C\log{\log{(T/C)}}}{T\log{(4\theta)}}\right)^{1/(4\theta-1)},
    \end{multline*}
    hence the desired result. The same reasoning applies to the second bound.
\end{proof}

This last result shows in particular that switching from an exponential restart scheme to a constant one only increases the overall complexity by a $\log\log$ factor.

\subsection{Robustness with respect to the number of restarts}

We now study the robustness of the complexity bound to misspecification of the number of restarts. The following result shows in particular that overestimating the optimal number of restarts $N^*$ is somewhat harmless. 

\begin{theorem}\label{th:optimality-ratio-bound}
    Let $C = 16c^{-4}\Delta_0^{1-4\theta}\rho L^2$, $\psi(x) = \Delta_0 \left(\frac{Cx}{T}\right)^{\frac{1-(4\theta)^{-x}}{4\theta-1}}$, and $N^* = \argmin_{x \in \reals_+}\psi(x)$. For $\tau > 0$, we have
    \begin{equation}\label{eq:psi_ratio}
        \frac{\psi(\tau N^*)}{\psi(N^*)} = 1 + \frac{N^*\log{(4\theta)}((4\theta)^{-N^*}+1) + (4\theta)^{-N^*} - 1}{2\left(4\theta-1\right)}(\log{\tau})^2 + o_{\tau\rightarrow 1}((\log{\tau})^2).
    \end{equation}
    Furthermore, for $\tau > 1$,
    \[
        \frac{\psi(\tau N^*)}{\psi(N^*)} < \tau^{1/(4\theta-1)}.
    \]
\end{theorem}

\begin{proof}
    Let $\tau > 0$. We have
    \begin{align*}
    \frac{\psi(\tau N^*)}{\psi(N^*)} &= \left(\frac{C\tau N^*}{T}\right)^{\frac{1-(4\theta)^{-\tau N^*}}{4\theta-1}} \left(\frac{CN^*}{T}\right)^{-\frac{1-(4\theta)^{- N^*}}{4\theta-1}} \\
    & = \exp{\left[ \frac{1-(4\theta)^{-\tau N^*}}{4\theta-1}\log{\tau} + \frac{(4\theta)^{-N^*}-(4\theta)^{-\tau N^*}}{4\theta-1}\log{\frac{CN^*}{T}} \right]} \\
    & = \exp{\left[ \frac{1-(4\theta)^{-\tau N^*}}{4\theta-1}\log{\tau} + \frac{(4\theta)^{-N^*}-(4\theta)^{-\tau N^*}}{4\theta-1} \cdot \frac{1-(4\theta)^{N^*}}{N^*\log{(4\theta)}} \right]} & \text{using~\eqref{eq:optimality-condition-constant-step}} \\
    & = \exp{\left[ \frac{1}{4\theta-1}\left((1-(4\theta)^{-\tau N^*})\log{\tau} - \frac{(1-(4\theta)^{-N^*})(1-(4\theta)^{N^*(1-\tau)})}{N^*\log{(4\theta)}} \right) \right]} \\
    & = \exp{\left[ \frac{1}{4\theta-1}\left((1-(4\theta)^{-e^\varepsilon N^*})\varepsilon - \frac{(1-(4\theta)^{-N^*})(1-(4\theta)^{N^*(1-e^\varepsilon)})}{N^*\log{(4\theta)}} \right) \right]} & \varepsilon = \log{\tau} \\
    & = 1 + \frac{N^*\log{(4\theta)}((4\theta)^{-N^*}+1) + (4\theta)^{-N^*} - 1}{2(4\theta-1)}\varepsilon^2 + o(\varepsilon^2), \\
    \end{align*}
    using a series expansion in $\varepsilon$. We now upper bound the above quantity for $\tau > 1$. First, note that $1-(4\theta)^{N^*(1-\tau)} > 0$. Hence, since $\exp$ is increasing and $\frac{1-(4\theta)^{-N^*}}{4\theta-1} > 0$,
    \[
        1 \leq \frac{\psi(\tau N^*)}{\psi(N^*)} < \exp{\biggl[ \frac{1-(4\theta)^{-\tau N^*}}{4\theta-1}\log{\tau}\biggr]} = \tau^{\frac{1-(4\theta)^{-\tau N^*}}{4\theta-1}} < \tau^{1/(4\theta-1)},
    \]
    hence the desired result.
\end{proof}

Note that the coefficient in front of $(\log \tau)^2$ in \eqref{eq:psi_ratio} is positive since $x\mapsto x(e^{-x}+1) +e^{-x} -1 >0$ for $x>0$.

\begin{corollary}
    We fix a total budget $T \in \mathbb{N}^*$ and define $C = 16c^{-4}\Delta_0^{1-4\theta}\rho L^2$ and $N^*$ as in \eqref{eq:Nstar-formulation}. For $\tau > 0$ sufficiently close to $1$, we have
    \begin{multline}
        \Expect[f_{1/(2\rho)}(x_{\tau N^*})]-f^* \\
        \underset{T\rightarrow+\infty}{\lesssim} \Delta_0 \left( \frac{C\log{\log{(T/C)}}}{T\log{(4\theta)}} \right)^{1/(4\theta-1)} \left( 1 + \frac{N^*\log{(4\theta)}((4\theta)^{-N^*}+1) + (4\theta)^{-N^*} - 1}{2(4\theta-1)}(\log{\tau})^2 \right).
    \end{multline}
    In addition, for $\tau > 1$, we have
    \BEQ
        \Expect[f_{1/(2\rho)}(x_{\tau N^*})]-f^* \underset{T\rightarrow+\infty}{\lesssim} \Delta_0 \left( \frac{C\log{\log{(T/C)}}}{T\log{(4\theta)}} \tau \right)^{1/(4\theta-1)}.
    \EEQ
\end{corollary}
\begin{proof}
    We obtain the result using \Cref{th:optimality-ratio-bound} and~\eqref{eq:equivalent-optimality-bound}.
\end{proof}

In \Cref{fig:adaptation}, we plot the bounds in \Cref{th:optimality-ratio-bound} for the regularity parameters
\[
\Delta_0 = 10, \qquad \theta = 0.75, \qquad \frac{16c^{-4}\rho L^2}{T} = 1.
\] 

\begin{figure}[ht]
    \centering
    \begin{minipage}[t]{0.75\textwidth}
        \centering
        \input{images/optimalityBound/optimalityBoundRatio_delta0_10_theta_0p75_eta_3_muT_1.tikz}
    \end{minipage}
    \caption{Robustness of the bound to the number of restarts $\tau N^*$, for $\Delta_0 = 10$, $\theta = 0.75$ and $16c^{-4}\rho L^2/T = 1$. The curves are the three quantities in \Cref{th:optimality-ratio-bound}: the exact ratio $\psi(\tau N^*)/\psi(N^*)$ between the bound with $\tau N^*$ restarts and the optimal bound (blue); its second-order expansion around $\tau = 1$ from~\eqref{eq:psi_ratio}, i.e., without the $o((\log\tau)^2)$ term (orange); and the upper bound $\tau^{1/(4\theta-1)}$, valid for $\tau > 1$ (green). Overestimating the number of restarts ($\tau > 1$) degrades the bound slowly, while underestimating it ($\tau < 1$) degrades it quickly.}
    \label{fig:adaptation}
\end{figure}

\subsection{Robustness with respect to step size misspecification}
We now establish a complexity bound on the optimality gap when the step size in~\eqref{eq:lr} is misspecified by a multiplicative factor $\zeta > 0$, i.e., when the step size becomes
\[
    \hat \alpha = \zeta \cdot \frac{(f_{1/(2\rho)}(x_k)-f^*)^{1/2}}{\sqrt{\rho} L \sqrt{t_k+1}}.
\]
The results below parallel those of \Cref{lem:step,th:primal-gap-complexity-bound}.

\begin{lemma}\label{lem:step-misspecified}
    Suppose $h$ is the indicator of a closed convex set $\mathcal{X}\subseteq\reals^d$ and $f_{1/(2\rho)}$ satisfies the K{\L} inequality in~\eqref{eq:KL} globally for some $\theta\in[1/2,1]$. Running \Cref{algo:prox-subgrad} with step size
    \BEQ
        \hat \alpha_k = \frac{\zeta (f_{1/(2\rho)}(x_k)-f^*)^{1/2}}{\sqrt{\rho} L \sqrt{t_k+1}}
    \EEQ
at each outer iteration $k$, for 
    \BEQ
        t_{k} \geq 16 c^{-4} \rho L^2 \Delta_0^{1-4\theta} \exp(4\theta \gamma) \exp((4\theta-1)\gamma k)
    \EEQ
    iterations ensures $\Expect[f_{1/(2\rho)}(x_{k})]-f^* \leq \left( \frac{\zeta^2 +1}{2\zeta} \right)^{\frac{1}{2\theta}\sum_{i=0}^{k-1}(4\theta)^{-i}} \Delta_0 \exp(-\gamma k)$.
\end{lemma}
\begin{proof}
We use the same arguments as in the proof of \Cref{lem:step}. Write $F \triangleq \frac{\zeta^2 +1}{2\zeta} = \frac{1}{2}\left( \zeta + \frac{1}{\zeta} \right) \geq 1$. We show by induction on $k$ that $\Expect[f_{1/(2\rho)}(x_{k})]-f^* \leq F^{\frac{1}{2\theta}\sum_{i=0}^{k-1}(4\theta)^{-i}} \Delta_0 \exp(-\gamma k)$ for every $k\geq0$. For $k=0$, $\Expect[f_{1/(2\rho)}(x_{0})]-f^*\leq\Delta_0$ holds with equality. Suppose now that it holds at some $k\geq0$, and run \Cref{algo:prox-subgrad} from $x_k$ with 
\[
\beta = \frac{\zeta (f_{1/(2\rho)}(x_k)-f^*)^{1/2}}{\sqrt{\rho} L}.
\]
\Cref{th:complexity} gives
\[
\Expect[\|\nabla f_{1/(2\rho)}(x_{k+1})\|^2]
\leq \frac{4F\sqrt{\rho}L(\Expect[f_{1/(2\rho)}(x_{k})]-f^*)^{1/2}}{\sqrt{t_k+1}}.
\]
Using the same reasoning as in the proof of \Cref{lem:step} yields
\begin{align*}
    \Expect[f_{1/(2\rho)}(x_{k+1})]-f^*
    & \leq c^{-1/\theta} \left(\frac{4F\sqrt{\rho}L(F^{\frac{1}{2\theta}\sum_{i=0}^{k-1}(4\theta)^{-i}} \Delta_0\exp{(-\gamma k)})^{1/2}}{\sqrt{t_k+1}}\right)^{1/(2\theta)} \\
    & \leq F^{\frac1{2\theta}} F^{\frac1{4\theta} \cdot \frac1{2\theta} \sum_{i=0}^{k-1}(4\theta)^{-i}} \Delta_0 \exp{(-\gamma(k+1))} \\
    & \leq F^{\frac{1}{2\theta}\sum_{i=0}^{k}(4\theta)^{-i}} \Delta_0\exp(-\gamma 
    (k+1)),
\end{align*}
which is the induction hypothesis at $k+1$ and completes the proof.
\end{proof}


\begin{theorem}\label{th:primal-gap-complexity-bound-misspecified-lr}
    Suppose $h$ is the indicator of a closed convex set $\mathcal{X}\subseteq\reals^d$, so that \Cref{algo:prox-subgrad} reduces to the projected stochastic subgradient method, and that $f_{1/(2\rho)}$ satisfies the K{\L} inequality in~\eqref{eq:KL} globally for some $\theta\in[1/2,1]$. Let $\zeta >0$ and let $T \triangleq \sum_{k=0}^{N-1} t_k$ be the total inner-iteration budget after $N$ outer iterations of the restart scheme in \Cref{algo:prox-subgrad-restart} with step size
    \[
        \hat \alpha_k = \frac{\zeta (f_{1/(2\rho)}(x_k)-f^*)^{1/2}}{\sqrt{\rho} L \sqrt{t_k+1}}.
    \]
    Then
    \BEQ
        \Expect[f_{1/(2\rho)}(x_N)]-f^* \leq \left( \frac{\zeta^2 +1}{2\zeta } \right)^{\frac{2}{4\theta -1}} \frac{f_{1/(2\rho)}(x_0)-f^*}{(\gamma(4\theta-1)C^{-1} T + 1)^{1/(4 \theta -1)}}
    \EEQ
    where $C=16 c^{-4} \rho L^2 \Delta_0^{1-4\theta} \exp(4\theta \gamma)$.
\end{theorem}
\begin{proof}
    We set $F = \frac{\zeta^2 +1}{2\zeta} \geq 1$. \Cref{lem:step-misspecified} shows that setting 
    \[
        t_{k} \geq 16 c^{-4} \rho L^2 \Delta_0^{1-4\theta} \exp(4\theta \gamma) \exp((4\theta-1)\gamma k)
    \]
    gives
    \[
        \Expect[f_{1/(2\rho)}(x_{k})]-f^* \leq F^{\frac{1}{2\theta}\sum_{i=0}^{k-1}(4\theta)^{-i}} \Delta_0 \exp(-\gamma k) \leq F^{\frac{1}{2\theta}\frac{1}{1-1/(4\theta)}} \Delta_0 \exp(-\gamma k) = F^{\frac{2}{4\theta -1}} \Delta_0 \exp(-\gamma k),
    \]
    since $F \geq 1$. \citet[Lem.\,2.1]{Roul17} then yields the desired result.
\end{proof}

This last result shows that misspecifying the step size by a multiplicative factor $\zeta$ degrades the complexity bound by at most
\[
\left( \frac{\zeta^2 +1}{2\zeta} \right)^{2},
\]
which means that the overall convergence rate remains unchanged. 

\section{Numerical experiments}\label{sec:numerical}
We present a series of experiments on the restarted proximal stochastic subgradient method. First, we examine the effect of the exponent $\theta$ in the K{\L} inequality~\eqref{eq:KL} on convergence rates. We then compare different restart schedules, as well as acceleration with and without restarts, on toy problems, and finally study restart sensitivity in neural network training.

\subsection{Effect of the K{\L} exponent on convergence rates}
According to \Cref{th:primal-gap-complexity-bound}, the restarted proximal subgradient method achieves the rate
\[
\Expect[f_{1/(2\rho)}(x_N)]-f^* = \mathcal{O}(T^{-1/(4\theta-1)}),
\]
where $T$ is the total number of inner iterations. The exponent ${1}/{(4\theta-1)}$ varies continuously from~$1$ when $\theta=1/2$ to $1/3$ when $\theta \to 1$, predicting slower convergence for larger values of $\theta$.

\subsubsection{Toy problems} 
To illustrate this phenomenon, we consider functions of the form
\[
f_1(x)=\|x-a\|^q,
\]
for several values of $q>2$. Since these functions satisfy the K{\L} inequality with exponent $\theta={(q-1)}/{q}$, increasing $q$ corresponds to increasing $\theta$. For each value of $q$, we run the restarted proximal subgradient method described in \Cref{algo:prox-subgrad-restart} and measure the optimality gap $f_{1,\lambda}(x_N)-f_1^*$ of the Moreau envelope for $\lambda = 1/(2\rho)$ as a function of the total number of inner iterations. We also consider functions of the form
\[
F_{p,\tau}(u,z) = 2\tau\left(|u| - \frac{u^2}{4} + \frac{\|z\|^p}{p} \right) + \iota_C (u,z), \qquad (u,z) \in \reals \times \reals^{d-1} \text{ and } C=[-1,1] \times \mathcal{B}^{d-1},
\]
for $\tau = 1$ and several values of $p>2$. Here, $\mathcal{B}^{d-1} \subset \reals^{d-1}$ denotes the closed unit ball centered at the origin, and $\iota_C$ denotes the indicator function of the closed convex set $C$. Similarly, these functions satisfy the K{\L} inequality with exponent $\theta={(p-1)}/{p}$. They are also $\tau$-weakly convex.

\begin{figure}[ht]
    \centering
    \begin{tabular}{@{}cc@{}}
        \begin{minipage}[t]{0.48\textwidth}
            \centering
            \input{images/power_norm/best_exponential_schedule_moreau_envelope_gap_comparison.tikz}
        \end{minipage}
        &
        \begin{minipage}[t]{0.48\textwidth}
            \centering
            \input{images/weakly_convex_pl/best_exponential_schedule_moreau_envelope_gap_comparison.tikz}
        \end{minipage}
    \end{tabular}
    \caption{Optimality gap of the Moreau envelope for the best exponential restart schedule, with step sizes defined as in~\eqref{eq:lr}. The dashed lines show the theoretical rates $T \mapsto T^{-1/(4\theta -1)}$, up to scaling. {\bf Left:} objectives of the form $f_1(x)=\|x-a\|^q$ for several values of $q$. {\bf Right:} objectives of the form $F_{p,\tau}(u,z) = 2\tau (|u| - u^2/4 + \|z\|^p/p) + \iota_C(u,z)$ for $\tau = 1$ and several values of $p$.}
    \label{fig:exponent-effect}
\end{figure}

\cref{fig:exponent-effect} shows that, in both sets of experiments, smaller values of $\theta$ lead to faster convergence. The empirical slopes closely match the theoretical complexity exponent ${1}/{(4\theta-1)}$, confirming that the K{\L} geometry strongly influences the efficiency of restart schemes.

\Cref{fig:numerical-theorem52-powers} compares \Cref{th:accel-no-restart-deterministic}, without restarts, on $f_q(x)=\|x-a\|^q$ for $q\in\{2,3,4,5,6\}$. The center, initial point, projection set, and noise seeds are shared across powers, while $c_q$, $L_q$, and $\Delta_{0,q}$ are recomputed for each objective. Writing $\mu_q=c_q^4/(32\rho L_q^2)$, the step sizes are given by
\begin{equation}
D_{t,q}=\left(\mu_q(4\theta_q-1)t+ \Delta_{0,q}^{-(4\theta_q-1)}\right)^{-1/(4\theta_q-1)},
\qquad
\alpha_{t,q}=\frac{c_q^2D_{t,q}^{2\theta_q}}{8\rho L_q^2}.
\label{eq:numerical-theorem52-prescription}
\end{equation}
We verify numerically that the assumptions of \Cref{th:accel-no-restart-deterministic} on the initial gap hold for every configuration. The theorem then guarantees that
\begin{equation}
\frac{\Expect[f_{q,\lambda}(x_T)-f_q^*]}{\Delta_{0,q}} 
\leq\frac{D_{T,q}}{\Delta_{0,q}}
=\mathcal{O}\!\left(T^{-1/(4\theta_q-1)}\right)
=\mathcal{O}\!\left(T^{-q/(3q-4)}\right).
\label{eq:numerical-theorem52-rate}
\end{equation}
Removing restarts therefore does not remove the dependence on the K{\L} exponent: it appears both in the implementable step size rule and in the guaranteed convergence rate. The predicted powers $-q/(3q-4)$ for $q=2,3,4,5,6$ are $1$, $3/5$, $1/2$, $5/11$, and $3/7$, respectively.

\begin{figure}[tbp]
    \centering
    \begin{minipage}{0.66\textwidth}
        \centering
        \input{images/power_norm_no_restart/no_restart_q_comparison.tikz}
    \end{minipage}
    \caption{The method of \Cref{th:accel-no-restart-deterministic} (without restarts) on $f_q(x)=\|x-a\|^q$, using $\theta_q=(q-1)/q$. Each mean gap is normalized by its own initial gap. Same-color dashed lines have slope $-q/(3q-4)$ and are anchored to the corresponding empirical curve at $T=2000$; they illustrate asymptotic rates, not the numerical upper bound $D_{T,q}/\Delta_{0,q}$. Thin dotted curves indicate one standard error.}
    \label{fig:numerical-theorem52-powers}
\end{figure}

The observed ordering is consistent with this geometric dependence: increasing $q$ gives progressively larger final relative gaps, from $2.75\times10^{-7}$ at $q=2$ to $3.54\times10^{-3}$ at $q=6$. However, a consistent ordering does not imply matching asymptotic rates. Several nonquadratic curves decrease more steeply than their reference lines over the displayed range. This is consistent with the guarantee being an upper bound and suggests that the worst-case analysis need not be tight for these objectives. Moreover, changing $q$ also changes the objective and the constants in~\eqref{eq:numerical-theorem52-prescription}, and normalization does not remove these effects. The figure thus demonstrates a persistent dependence on problem geometry without restarts, but does not isolate the rate exponent from all transient effects.

\subsubsection{Deep matrix factorization}
Here, we test our algorithms on a deep matrix factorization problem. Given a target $M\in\reals^{d_L\times d_0}$ of rank $r$, we seek factors $W_i\in\reals^{d_i\times d_{i-1}}$, $i=1,\dots,L$, solving
\begin{equation}\label{eq:dmf}
\min_{W=(W_1,\dots,W_L)}\ F(W)\triangleq\phi\big(P(W)-M\big),\qquad P(W)\triangleq W_L\cdots W_1 .
\end{equation}
For $L=2$ and $\phi=\norm{\cdot}_\F^2$, this is the Burer--Monteiro factorization. For $L\ge3$, it is a deep linear network trained on $M$. For $\phi=\norm{\cdot}_1$, it is the robust factorization used in robust PCA and outlier-corrupted matrix sensing. In the stochastic setting, $\phi$ is an average over sampled entries or measurements, $\phi(Z)=\frac1m\sum_{j}\varphi(\langle A_j,Z\rangle)$, and only stochastic subgradients are available.
\vskip 1ex
\paragraph{Loss.} Throughout, $\phi:\reals^{m\times n}\to\reals_+$ is convex, positively homogeneous of degree $q\ge1$, i.e., $\phi(tZ)=t^q\phi(Z)$ for $t>0$, and $C^1$ on $\{\phi>0\}$. The two cases of interest are the smooth squared loss $\phi=\norm{\cdot}_\F^2$ ($q=2$) and the nonsmooth robust losses $\phi=\norm{\cdot}_\F$ or $\phi=\norm{\cdot}_1$ ($q=1$). 
\vskip 1ex
\paragraph{Exact rank versus overparameterization.} When the inner widths equal $r$ and $M$ is well conditioned, the minimizers of~\eqref{eq:dmf} are nondegenerate: the squared loss satisfies a K{\L} inequality with exponent $1/2$ and the $\ell_1$ loss is sharp, so gradient-type methods converge linearly. When the inner widths are $k>r$, every minimizer has $k-r$ singular directions along which the factors can be shrunk to zero without changing $P$. Along these directions, the residual is a product of $L$ small quantities, hence of order $L$ in the factors: the landscape is flat, the linear rates disappear, and the relevant regularity is a K{\L} inequality with an exponent strictly between $1/2$ and $1$. 
\vskip 1ex
\paragraph{K{\L} inequality.} 
The objective in~\eqref{eq:dmf} is semialgebraic, hence subanalytic, so the result of \citet{Bolt07} guarantees that a K{\L} inequality holds, but with an unknown exponent. In the scalar case, however, we can obtain an explicit expression for the K{\L} exponent. Let $p(w)=w_1w_2\cdots w_L$ on $\reals^L$ and $f(w)=\phi(p(w))$ with $\phi(r)=|r|^q$. We first show a slight generalization of the result in \citet[App.\,C]{Masi24}.

\begin{proposition}[Scalar K{\L} bound]\label{prop:scalar}
For every $w\in\reals^L$ with $p(w)\ne0$,
\begin{equation}\label{eq:scalar}
\norm{\nabla f(w)} \;\ge\; q\sqrt{L}\; f(w)^{\,1-\frac{1}{qL}},
\end{equation}
with equality if and only if $|w_1|=\dots=|w_L|$. Consequently, the K{\L} exponent of $f$ at the origin is exactly $\theta=1-1/(qL)$, with the global, location-free constant $c=q\sqrt L$.
\end{proposition}
\begin{proof}
Since $\partial_i p = p/w_i$, the chain rule gives $\partial_i f = q|p|^{q-1}\sgn(p)\,p/w_i = q|p|^q/w_i$, hence
\[
\norm{\nabla f}^2 = q^2|p|^{2q}\sum_{i=1}^L \frac1{w_i^2}
\;\ge\; q^2|p|^{2q}\, L\Big(\prod_{i=1}^L\frac1{w_i^2}\Big)^{1/L}
= q^2L\,|p|^{2q-2/L},
\]
by the AM--GM inequality, with equality if and only if all $w_i^{-2}$ coincide. Taking square roots and using $|p|^q=f$ yields \eqref{eq:scalar}. On the ray $w=(t,\dots,t)$, $f=t^{qL}$ and $\norm{\nabla f}=q\sqrt L\,t^{qL-1}=q\sqrt L f^{1-1/(qL)}$, so no exponent $\theta'<\theta$ can hold near $0$: $f^{\theta'}/f^{\theta}\to\infty$ as $f\to0$.
\end{proof}

No such result is available in the general overcomplete matrix case, but our numerical results suggest that the same exponent applies, at least in the instances considered here.

\vskip 1ex
\paragraph{Weak convexity.} 
We can also show that the objective in~\eqref{eq:dmf} is weakly convex on bounded sets.

\begin{proposition}[Weak convexity of the composite loss]\label{prop:wc}
Let $\phi$ be convex and $\ell$-Lipschitz (e.g., $\phi=\norm{\cdot}_\F$, $\ell=1$; $\phi=\norm{\cdot}_1$, $\ell=\sqrt{mn}$ w.r.t.\ the Frobenius norm), and let $\mathcal X_R=\{W:\norm{W_i}_\F\le R,\ i=1,\dots,L\}$. Then $F=\phi\circ P$ is $\rho$-weakly convex on $\mathcal X_R$ with
\begin{equation}\label{eq:rho}
\rho\ \le\ \ell\,(L-1)\,R^{L-2}.
\end{equation}
\end{proposition}
\begin{proof}
The second derivative of $P$ is\[D^2P(W)[\dot W,\ddot W]=\sum_{i\ne j}W_{j+1:L}\,\dot W_j\,W_{i+1:j-1}\,\ddot W_i\,W_{1:i-1}\](summing over both orderings of $i,j$). On $\mathcal X_R$ each term is bounded by $R^{L-2}\norm{\dot W_j}_\F\norm{\ddot W_i}_\F$, and $\sum_{i\ne j}a_jb_i=a^\top(J-I)b\le(L-1)\norm a\norm b$ for $a,b\ge0$ since the all-ones matrix $J-I$ has largest eigenvalue $L-1$. Hence $\norm{D^2P(W)}\le(L-1)R^{L-2} \triangleq \beta_R$ and $\nabla P$ is $\beta_R$-Lipschitz there. Fix $x,y\in\mathcal X_R$ and $v\in\partial\phi(P(x))$, and write $T \triangleq P(x)+\nabla P(x)(y-x)$, so that $\norm{P(y)-T}\le\frac{\beta_R}2\norm{y-x}^2$ (Taylor). Using first the Lipschitz continuity of $\phi$, then its convexity,
\begin{align*}
\phi(P(y)) &\ \ge\ \phi(T)-\ell\,\norm{P(y)-T}\\
&\ \ge\ \phi(P(x))+\langle\nabla P(x)^\ast v,\ y-x\rangle-\frac{\ell\beta_R}2\norm{y-x}^2,
\end{align*}
which is the weak-convexity inequality with $\rho=\ell\beta_R$ and subgradient $\nabla P(x)^\ast v$. 
\end{proof}

\vskip 1ex
\paragraph{Numerical experiments} 
We let $F(W)=\frac1m\sum_{j=1}^m\bigl|\langle A_j,\,W_L\cdots W_1\rangle-b_j\bigr|$ with $A_j\in\mathbb R^{20\times20}$ Gaussian, $m=800$, $b_j=\langle A_j,M\rangle$, $M=U_r\Sigma V_r^\top$, $r=2$, $\sigma_1=\sigma_2=1$. The measurements are noise-free and $M$ is attainable, so $F^\star=0$ exactly and $F$ is sharp in $P=W_L\cdots W_1$. Widths $k=4>r$: $W_1\in\mathbb R^{4\times20}$, $W_2,\dots,W_{L-1}\in\mathbb R^{4\times4}$, $W_L\in\mathbb R^{20\times4}$, two degenerate directions, K{\L} exponent $\theta=1-1/L$ ($q=1$). Constraint $\|W_i\|_F\le R=1.5\max_i\|W_i^0\|_F$. Initialisation in $P$-space, identical for every $L$: $P_0=U_kDV_k^\top$, $D=\mathrm{diag}(1\pm0.3,\,1\pm0.3,\,0.3,\,0.3)$, balanced factors $W_i^0$, so $F(W_0)=0.46$ for all $L$. Ten seeds ($A,M,b$, signs) paired across $L$. 

For the oracle we pick a uniform $j$, chain-rule subgradient of $|\langle A_j,P\rangle-b_j|$. Since $M\neq0$ the noise does not vanish at the optimum (a constant step floors at $6\cdot10^{-2}$), so step decay is necessary. Epoch $k$ runs $t_k=t_0e^{(4\theta-1)\gamma k}$ projected stochastic subgradient steps of size $\alpha_k=\sqrt{\Delta_0}\,e^{-\gamma k/2}/(\sqrt\rho\,\hat L\sqrt{t_k+1})$ (cf. eqs.~\eqref{eq:lr} and~\eqref{eq:tk}) and restarts from the last iterate, with $\theta=1-1/L$ and $\gamma=1/(4\theta)$. Constants are measured at $W_0$: $\hat L^2=\mathbb E\|G\|^2$, $\hat\rho=\hat\ell L(L-1)R^{L-2}$ ($2.1\to350$ for $L=2\to6$), $\Delta_0=F(W_0)$. 

The first epoch is $t_0=4\,t_{\rm nat}$ with $t_{\rm nat}=\hat L^2/\|\nabla F(W_0)\|^2$ ($=116,87,72,62,56$), the sample count at which the
averaged subgradient's noise equals its mean; the paper's $t_0$ of~\eqref{eq:tk} is $10^2$--$2\cdot10^5$ times longer, while $4t_{\rm nat}$ already gives the contraction $e^{-\gamma}$ per epoch. Budget $10^8$ samples ($\approx20$ epochs), geometric mean of $F(W_T)$ over the 10 seeds, slopes fitted on the last decade. \Cref{th:primal-gap-complexity-bound} predicts $F(W_T)-F^\star\lesssim T^{-1/(4\theta-1)}$.

\begin{figure}[ht]
    \centering
    \begin{minipage}{0.66\textwidth}\centering\input{images/dmf/exp1_matrix_gap.tikz}\end{minipage}
    \caption{Algorithm~\ref{algo:prox-subgrad-restart-grad} on the rank-deficient deep matrix factorization instance ($n=20$, $k=4$, $r=2$, $m=800$, $\ell_1$ loss) for depths $L=2,\dots,6$, with the $\theta$-matched schedule $\theta=1-1/L$, $\gamma=1/(4\theta)$, $t_0=4\,t_{\rm nat}$. Gap $F(W_T)-F^\star$ (here $F^\star=0$) against the number of samples $T$, geometric mean over 10 paired instances. Dashed: the rate $T^{-1/(4\theta-1)}$ of \Cref{th:primal-gap-complexity-bound}, anchored at the last point. The legend gives the slope fitted on the last decade and the predicted slope $-1/(4\theta-1)$.\label{fig:placeholder}}
\end{figure}


We observe that the rate indeed degrades with depth in the order predicted by $\theta=1-1/L$, and the depth ordering of the curves is already visible after $10^3$ samples. The alignment between predicted convergence rates and empirical ones seems to support our conjecture that the K{\L} exponent of problem~\eqref{eq:dmf} follows that of Proposition~\ref{prop:scalar} even in the matrix case, at least in this particular instance. Overall here, deeper linear networks appear to have flatter landscapes, which directly impacts complexity.

\subsection{Effect of the restart schedule}
Next, we investigate the influence of the restart schedule on convergence. For a fixed total iteration budget $T$, the restart scheme requires choosing both the number of outer iterations $N$ and the corresponding numbers of inner iterations $t_k$. This choice is nontrivial: restarting too often may prevent the inner method from making enough progress between successive restarts, while restarting too rarely reduces the algorithm to a standard subgradient method with little adaptation to the K{\L} geometry of the problem. 

We first consider the constant restart scheme, where each outer iteration is run for the same number of inner iterations, $t_k = t$, for $k=0,\ldots,N-1$, so that $T = Nt$. This setting is analyzed in \Cref{th:constant-restart}, which shows that the objective gap satisfies a bound of the form
\[
\Expect[ f_\lambda(x_N) ] - f^* \leq \Delta_0 \left( \frac{CN}{T} \right)^{\frac{1-(4\theta)^{-N}}{4\theta-1}}.
\]
This bound predicts the existence of an optimal number of restarts $N^*$ for a fixed budget $T$.

\begin{figure}[ht]
    \centering
    \begin{tabular}{@{}cc@{}}
        \begin{minipage}[t]{0.48\textwidth}
            \centering
\begin{tikzpicture}
\begin{axis}[xlabel={Number of restarts}, ylabel={$f_{2,\lambda}(x_N) - f_2^*$}, legend pos={north east}, legend cell align={left}, grid={major}, width={\textwidth}, height={0.8\textwidth}, unbounded coords={discard}]
    \addplot[color={T-Q-B1}, line width={0.7pt}, no marks, dashed, forget plot]
        table[row sep={\\}]
        {
            \\
            0.0  9.352801511433413e-7  \\
            1.0  3.670433158728845e-7  \\
            4.0  2.2312764676211196e-7  \\
            9.0  1.423692443052757e-7  \\
            19.0  1.1354930160594023e-7  \\
            49.0  1.0217168431436271e-7  \\
            99.0  1.297007828994287e-7  \\
            199.0  1.9852448448127003e-7  \\
            499.0  3.680367145855918e-7  \\
        }
        ;
    \addplot[color={T-Q-B1}, line width={0.7pt}, no marks, dashed, forget plot]
        table[row sep={\\}]
        {
            \\
            0.0  1.0391884060463919e-7  \\
            1.0  1.0553488896248768e-7  \\
            4.0  6.619466157122763e-8  \\
            9.0  1.293859481726693e-8  \\
            19.0  2.716955611588191e-8  \\
            49.0  1.3469567992978635e-8  \\
            99.0  7.621797074468813e-9  \\
            199.0  2.220446049250313e-16  \\
            499.0  2.220446049250313e-16  \\
        }
        ;
    \addplot[color={T-Q-B1}, line width={1.1pt}, mark={*}]
        table[row sep={\\}]
        {
            \\
            0.0  5.195994958739902e-7  \\
            1.0  2.3628910241768609e-7  \\
            4.0  1.446611541666698e-7  \\
            9.0  7.765391956127132e-8  \\
            19.0  7.035942886091107e-8  \\
            49.0  5.7820626153670675e-8  \\
            99.0  6.866128998694876e-8  \\
            199.0  9.668926637118602e-8  \\
            499.0  1.606306206269892e-7  \\
        }
        ;
    \addlegendentry {Mean $\pm$ std}
    \addplot[only marks, mark={*}, color={T-Q-B3}, mark size={2.2pt}]
        table[row sep={\\}]
        {
            \\
            49.0  5.7820626153670675e-8  \\
        }
        ;
    \addlegendentry {Best mean}
    \addplot[only marks, mark={*}, color={black}, mark size={2.2pt}]
        table[row sep={\\}]
        {
            \\
            0.0  5.195994958739902e-7  \\
        }
        ;
    \addlegendentry {No restart}
\end{axis}
\end{tikzpicture}
        \end{minipage}
        &
        \begin{minipage}[t]{0.48\textwidth}
            \centering
\begin{tikzpicture}
\begin{axis}[xlabel={Number of restarts}, ylabel={$F_\lambda(x_N) - F^*$}, legend pos={north west}, legend cell align={left}, grid={major}, width={\textwidth}, height={0.8\textwidth}, unbounded coords={discard}, ymode={log}]
    \addplot[color={T-Q-B1}, line width={0.7pt}, no marks, dashed, forget plot]
        table[row sep={\\}]
        {
            \\
            0.0  4.562761765993435e-5  \\
            9.0  3.221299685965245e-12  \\
            19.0  3.275172946236612e-12  \\
            49.0  1.1819341191127047e-10  \\
            99.0  0.0004762911214454607  \\
            199.0  3.181104096437011  \\
        }
        ;
    \addplot[color={T-Q-B1}, line width={0.7pt}, no marks, dashed, forget plot]
        table[row sep={\\}]
        {
            \\
            0.0  1.8737490525081626e-6  \\
            9.0  2.5052258369176958e-12  \\
            19.0  2.2724429323235583e-12  \\
            49.0  2.0049255297697583e-11  \\
            99.0  0.00023034970473263656  \\
            199.0  0.0020336796204617146  \\
        }
        ;
    \addplot[color={T-Q-B1}, line width={1.1pt}, mark={*}]
        table[row sep={\\}]
        {
            \\
            0.0  9.246334698598505e-6  \\
            9.0  2.840789186429537e-12  \\
            19.0  2.7281208942810275e-12  \\
            49.0  4.86794606576017e-11  \\
            99.0  0.0003312303114024111  \\
            199.0  0.08043224833045033  \\
        }
        ;
    \addplot[only marks, mark={*}, color={T-Q-B3}, mark size={2.2pt}]
        table[row sep={\\}]
        {
            \\
            19.0  2.7281208942810275e-12  \\
        }
        ;
    \addplot[only marks, mark={*}, color={black}, mark size={2.2pt}]
        table[row sep={\\}]
        {
            \\
            0.0  9.246334698598505e-6  \\
        }
        ;
\end{axis}
\end{tikzpicture}
        \end{minipage}
    \end{tabular}
    \caption{Final optimality gap of the Moreau envelope for several constant restart schedules. {\bf Left:} objective of the form $f_2(x)= \sum_i \|x-a_i\|^2$. {\bf Right:} objective of the form $F(u,z) = 2 (|u| - u^2/4 + \|z\|^{2.5}/2.5) + \iota_C(u,z)$.}
    \label{fig:constant-restart-grid}
\end{figure}

\Cref{fig:constant-restart-grid} reports the mean and standard deviation, across $10$ runs, of the final optimality gap of the Moreau envelope for several values of $t$, with the same total iteration budget $T = 10{,}000$. We consider the objectives $f_2(x)= \sum_{i=1}^n \|x-a_i\|^2$ and $F(u,z) = 2 (|u| - u^2/4 + \|z\|^{2.5}/2.5) + \iota_C(u,z)$, with step sizes defined as in~\eqref{eq:lr}. In both our theoretical framework and practical implementations, we do not have direct access to the exact gradient of the objective function. Instead, the algorithms rely on a stochastic oracle that provides an unbiased estimate of the gradient. For $f_2$, rather than computing the full gradient, we follow the standard stochastic gradient model by sampling an index $i$ uniformly at random and using the estimator 
\[
2n(x-a_i),
\]
which explains the noise observed in the figure. For $F$, we introduce multiplicative Gaussian noise $\xi \sim \mathcal{N}(0,\sigma^2I)$ with $\sigma = 0.01$, yielding the stochastic gradient estimator
\[
G(x,\xi) = (1+\xi) \odot \nabla F(x) = \left( \left( 1+\xi_i \right) \frac{\partial F(x)}{\partial x_i} \right)_{i=1}^d.
\]

For $f_2$, the best performance is obtained with $49$ restarts ($N = 50$, $t = 200$). For $F$, the best performance is obtained with $4$ restarts ($N = 5$, $t = 2000$). The setting without restarts ($t = T$) converges more slowly in comparison, indicating that step size adaptation is beneficial. However, very frequent restarts do not provide further improvement and can be less effective than an intermediate choice. Overall, the results show that restarting helps the algorithm adapt to the changing local scale of the problem, and that there is a fairly robust optimal range for the number of inner iterations, as shown in \Cref{sec:adapt}.

We then compare the constant restart schedule with an exponential restart schedule of the form
\[
t_k = t_0 \exp((4\theta-1)\gamma k).
\]
This choice is motivated by the theoretical restart schedule in \Cref{lem:step}, where the number of inner iterations grows exponentially with the outer iteration index.

\begin{figure}[ht]
    \centering
    \begin{tabular}{@{}cc@{}}
        \begin{minipage}[t]{0.48\textwidth}
            \centering
\begin{tikzpicture}
\begin{axis}[xlabel={Initial schedule length}, ylabel={$f_{2,\lambda}(x_N) - f_2^*$}, legend pos={north east}, legend cell align={left}, grid={major}, width={\textwidth}, height={0.8\textwidth}, unbounded coords={discard}, xmode={log}]
    \addplot[color={T-Q-B1}, line width={0.7pt}, no marks, dashed, forget plot]
        table[row sep={\\}]
        {
            \\
            2.0  1.409118148705867e-6  \\
            5.0  6.469438423475339e-7  \\
            10.0  4.245716332286818e-7  \\
            20.0  3.680367145855918e-7  \\
            50.0  1.9852448448127003e-7  \\
            100.0  1.297007828994287e-7  \\
            200.0  1.0217168431436271e-7  \\
            500.0  1.1354930160594023e-7  \\
            1000.0  1.423692443052757e-7  \\
            2000.0  2.2312764676211196e-7  \\
            5000.0  3.670433158728845e-7  \\
            10000.0  9.352801511433413e-7  \\
        }
        ;
    \addplot[color={T-Q-B1}, line width={0.7pt}, no marks, dashed, forget plot]
        table[row sep={\\}]
        {
            \\
            2.0  2.220446049250313e-16  \\
            5.0  2.220446049250313e-16  \\
            10.0  2.220446049250313e-16  \\
            20.0  2.220446049250313e-16  \\
            50.0  2.220446049250313e-16  \\
            100.0  7.621797074468813e-9  \\
            200.0  1.3469567992978635e-8  \\
            500.0  2.716955611588191e-8  \\
            1000.0  1.293859481726693e-8  \\
            2000.0  6.619466157122763e-8  \\
            5000.0  1.0553488896248768e-7  \\
            10000.0  1.0391884060463919e-7  \\
        }
        ;
    \addplot[color={T-Q-B1}, line width={1.1pt}, mark={*}]
        table[row sep={\\}]
        {
            \\
            2.0  6.449390895113538e-7  \\
            5.0  3.125105408230411e-7  \\
            10.0  1.8408015947102286e-7  \\
            20.0  1.606306206269892e-7  \\
            50.0  9.668926637118602e-8  \\
            100.0  6.866128998694876e-8  \\
            200.0  5.7820626153670675e-8  \\
            500.0  7.035942886091107e-8  \\
            1000.0  7.765391956127132e-8  \\
            2000.0  1.446611541666698e-7  \\
            5000.0  2.3628910241768609e-7  \\
            10000.0  5.195994958739902e-7  \\
        }
        ;
    \addlegendentry {Constant}
    \addplot[only marks, mark={*}, color={T-Q-B3}, mark size={2.2pt}, forget plot]
        table[row sep={\\}]
        {
            \\
            200.0  5.7820626153670675e-8  \\
        }
        ;
    \addplot[only marks, mark={*}, color={black}, mark size={2.2pt}]
        table[row sep={\\}]
        {
            \\
            10000.0  5.195994958739902e-7  \\
        }
        ;
    \addlegendentry {No restart}
    \addplot[color={T-Q-B5}, line width={0.7pt}, no marks, dashed, forget plot]
        table[row sep={\\}]
        {
            \\
            1.0  7.179758787802382e-8  \\
            2.0  7.020030983240742e-8  \\
            5.0  7.503086584604763e-8  \\
            10.0  7.63044591759798e-8  \\
            20.0  7.726587170497419e-8  \\
            50.0  8.484431110757422e-8  \\
            100.0  9.148846074450186e-8  \\
        }
        ;
    \addplot[color={T-Q-B5}, line width={0.7pt}, no marks, dashed, forget plot]
        table[row sep={\\}]
        {
            \\
            1.0  5.976657832292067e-9  \\
            2.0  7.840160609057592e-9  \\
            5.0  8.566543326805444e-9  \\
            10.0  9.912912910159424e-9  \\
            20.0  1.011839629439442e-8  \\
            50.0  1.0358352778055367e-8  \\
            100.0  1.155265054339185e-8  \\
        }
        ;
    \addplot[color={T-Q-B5}, line width={1.1pt}, mark={*}]
        table[row sep={\\}]
        {
            \\
            1.0  3.888712285515794e-8  \\
            2.0  3.9020235220732505e-8  \\
            5.0  4.179870458642654e-8  \\
            10.0  4.310868604306961e-8  \\
            20.0  4.36921339996843e-8  \\
            50.0  4.76013319428148e-8  \\
            100.0  5.152055564394686e-8  \\
        }
        ;
    \addlegendentry {Exponential}
    \addplot[only marks, mark={*}, color={T-Q-B3}, mark size={2.2pt}]
        table[row sep={\\}]
        {
            \\
            1.0  3.888712285515794e-8  \\
        }
        ;
    \addlegendentry {Best mean}
\end{axis}
\end{tikzpicture}
        \end{minipage}
        &
        \begin{minipage}[t]{0.48\textwidth}
            \centering
\begin{tikzpicture}
\begin{axis}[xlabel={Initial schedule length}, ylabel={$F_\lambda(x_N) - F^*$}, legend pos={north west}, legend cell align={left}, grid={major}, width={\textwidth}, height={0.8\textwidth}, unbounded coords={discard}, xmode={log}, ymode={log}]
    \addplot[color={T-Q-B1}, line width={0.7pt}, no marks, dashed, forget plot]
        table[row sep={\\}]
        {
            \\
            50.0  3.181104096437011  \\
            100.0  0.0004762911214454607  \\
            200.0  1.1819341191127047e-10  \\
            500.0  3.275172946236612e-12  \\
            1000.0  3.221299685965245e-12  \\
            10000.0  4.562761765993435e-5  \\
        }
        ;
    \addplot[color={T-Q-B1}, line width={0.7pt}, no marks, dashed, forget plot]
        table[row sep={\\}]
        {
            \\
            50.0  0.0020336796204617146  \\
            100.0  0.00023034970473263656  \\
            200.0  2.0049255297697583e-11  \\
            500.0  2.2724429323235583e-12  \\
            1000.0  2.5052258369176958e-12  \\
            10000.0  1.8737490525081626e-6  \\
        }
        ;
    \addplot[color={T-Q-B1}, line width={1.1pt}, mark={*}]
        table[row sep={\\}]
        {
            \\
            50.0  0.08043224833045033  \\
            100.0  0.0003312303114024111  \\
            200.0  4.86794606576017e-11  \\
            500.0  2.7281208942810275e-12  \\
            1000.0  2.840789186429537e-12  \\
            10000.0  9.246334698598505e-6  \\
        }
        ;
    \addplot[only marks, mark={*}, color={T-Q-B3}, mark size={2.2pt}, forget plot]
        table[row sep={\\}]
        {
            \\
            500.0  2.7281208942810275e-12  \\
        }
        ;
    \addplot[only marks, mark={*}, color={black}, mark size={2.2pt}]
        table[row sep={\\}]
        {
            \\
            10000.0  9.246334698598505e-6  \\
        }
        ;
    \addplot[color={T-Q-B5}, line width={0.7pt}, no marks, dashed, forget plot]
        table[row sep={\\}]
        {
            \\
            1.0  1.126651229028058e-9  \\
            3.0  6.05492383351898e-10  \\
            10.0  5.687157923358989e-11  \\
            30.0  4.4405446752582434e-11  \\
            100.0  9.627928035335546e-12  \\
        }
        ;
    \addplot[color={T-Q-B5}, line width={0.7pt}, no marks, dashed, forget plot]
        table[row sep={\\}]
        {
            \\
            1.0  1.17464142543804e-11  \\
            3.0  5.323395268145044e-11  \\
            10.0  1.3910431689346502e-12  \\
            30.0  1.1120921049616268e-12  \\
            100.0  2.2458964548983765e-12  \\
        }
        ;
    \addplot[color={T-Q-B5}, line width={1.1pt}, mark={*}]
        table[row sep={\\}]
        {
            \\
            1.0  1.1503961081458151e-10  \\
            3.0  1.7953482359791256e-10  \\
            10.0  8.89442644578114e-12  \\
            30.0  7.027300103940406e-12  \\
            100.0  4.650089186518555e-12  \\
        }
        ;
    \addplot[only marks, mark={*}, color={T-Q-B3}, mark size={2.2pt}]
        table[row sep={\\}]
        {
            \\
            100.0  4.650089186518555e-12  \\
        }
        ;
\end{axis}
\end{tikzpicture}
        \end{minipage}
    \end{tabular}
    \caption{Final optimality gap of the Moreau envelope for several constant and exponential restart schedules. The x-axis indicates the number of inner iterations $t_0$ in the first outer iteration, which stays the same for the constant restart schedule and increases exponentially for the exponential restart schedule. {\bf Left:} objective of the form $f_2(x)= \sum_i \|x-a_i\|^2$. {\bf Right:} objective of the form $F(u,z) = 2 (|u| - u^2/4 + \|z\|^{2.5}/2.5) + \iota_C(u,z)$.}
    \label{fig:schedule-comparison}
  \end{figure}

\cref{fig:schedule-comparison} compares the constant and exponential restart schedules. As before, we report the final optimality gap of the Moreau envelope, define the step sizes as in~\eqref{eq:lr}, and use the same stochastic gradient oracles. Furthermore, we set $\gamma = 1/(4\theta)$. Instead of plotting the gap as a function of the number of restarts as in \cref{fig:constant-restart-grid}, we plot it as a function of $t_0$, the number of inner iterations in the first outer iteration, which is more meaningful for the exponential restart schedule. This figure illustrates the practical effect of the restart schedule under a fixed computational budget. Both constant and exponential restart schedules improve over the baseline without restarts, confirming that adapting the step size through restarts can lead to better finite-budget performance. The constant restart schedule performs well when its restart frequency is carefully tuned, but its performance is more sensitive to the number of inner iterations. In contrast, the exponential restart schedule appears more robust across a wider range of values of $t_0$, since it combines frequent step size adaptation early on with progressively longer outer iterations.

\subsection{Comparison of acceleration with and without restarts}
We next compare the restart scheme of \Cref{algo:prox-subgrad-restart}, whose step sizes use the observed optimality gap, the restart scheme with deterministic step sizes of \Cref{algo:prox-subgrad-restart-grad}, and the method without restarts of \Cref{th:accel-no-restart-deterministic}. We consider $f_q(x)=\|x-a\|^q$ in dimension $d=10$, with $q\geq2$. On this ball, the Moreau envelope satisfies the K{\L} inequality with exponent $\theta_q=(q-1)/q$. The stochastic oracle returns $G_q(x,\xi)=\nabla f_q(x)+(\sigma/\sqrt d)\xi$, where $\xi\sim\mathcal N(0,I_d)$ and $\sigma=0.1$, so that $\Expect\|G_q(x,\xi)-\nabla f_q(x)\|^2=\sigma^2$. Each curve averages 20 runs over $T=20{,}000$ gradient calls, using paired noise seeds across the compared configurations. We report $\Expect[f_{q,\lambda}(x_T)-f_q^*]/\Delta_{0,q}$, where $\Delta_{0,q}=f_{q,\lambda}(x_0)-f_q^*$ and $f_q^*=0$.

\Cref{fig:numerical-variants-misspecification} fixes $q=3$, hence $\theta_q=2/3$, and varies the assumed exponent $\widehat\theta\in\{0.5,0.6,2/3,0.8,0.95\}$. For each assumed exponent, \Cref{algo:prox-subgrad-restart,algo:prox-subgrad-restart-grad} use the same numbers of inner iterations $t_k\approx t_0\exp((4\widehat\theta-1)\gamma k)$, with $\gamma=0.25$. These numbers are rounded to integers, and the last outer iteration is shortened to exhaust the budget. The step sizes of the two methods differ only through the gap used in the numerator:
\begin{equation}
\alpha_k^{(2)}= \frac{\sqrt{f_{q,\lambda}(x_k^{(2)})-f_q^*}}{\sqrt\rho L\sqrt{t_k+1}},
\qquad
\alpha_k^{(3)}= \frac{\sqrt{\Delta_{0,q}}e^{-\gamma k/2}}{\sqrt\rho L\sqrt{t_k+1}}.
\label{eq:numerical-variant-steps}
\end{equation}
Here, $k$ indexes outer iterations. For the method of \Cref{th:accel-no-restart-deterministic}, we substitute $\widehat\theta$ for $\theta$ in its step size rule~\eqref{eq:accel-step-deterministic}, keeping the same problem-dependent constants $c$, $L$, $\rho$, and $\Delta_{0,q}$ across assumed exponents. Note that ``deterministic'' refers to the step size rule; all three methods use stochastic gradients.

With the correct exponent, the final mean relative gaps are approximately $5.50\times10^{-4}$ for \Cref{algo:prox-subgrad-restart}, $1.04\times10^{-5}$ for \Cref{algo:prox-subgrad-restart-grad}, and $7.65\times10^{-5}$ for the method of \Cref{th:accel-no-restart-deterministic}. Thus \Cref{algo:prox-subgrad-restart-grad} attains a gap about 53 times smaller than \Cref{algo:prox-subgrad-restart} despite using exactly the same restart schedule. 


\begin{figure}[tbp]
    \centering
    \resizebox{0.75\linewidth}{!}{%
        \input{images/power_norm_no_restart/q3_T20000.tikz}%
    }
    \caption{Comparison on $f_3(x)=\|x-a\|^3$ with true K{\L} exponent $2/3$. Solid, dash-dotted, and dashed curves represent \Cref{algo:prox-subgrad-restart,algo:prox-subgrad-restart-grad} and the method of \Cref{th:accel-no-restart-deterministic}, respectively. Matching colors indicate matching assumed exponents. Both restart schemes use the same fixed prefactor $t_0=10$. The black reference has slope $-3/5$. All runs share the same initial point and normalized initial gap of one; the logarithmic axis begins after the first gradient call.}
    \label{fig:numerical-variants-misspecification}
\end{figure}

\subsection{Restart sensitivity in neural network training}
We also evaluate the restarted proximal stochastic subgradient method on neural network training problems. We consider two models: a nanoGPT model trained on the works of Shakespeare and a ResNet18 model trained on CIFAR-10. For nanoGPT, the architecture consists of $2$ layers, $2$ attention heads, and an embedding dimension of $128$. These experiments test whether the restart behavior predicted by the theory remains visible in these nonconvex, high-dimensional learning problems.

In the theoretical restart schemes, the step size depends on constants such as the K{\L} exponent, the weak convexity parameter, and the noise level. These constants are not available in neural network training. We therefore introduce a multiplicative step size factor $a>0$, which acts as an empirical estimate of the unknown theoretical scaling, as well as an optional estimate~$\hat f^*$ of the minimum value of the objective. The step size then becomes
\[
    \alpha_k = a \cdot \frac{(f(x_k) - \hat f^*)^{1/2}}{\sqrt{t_k + 1}}
\]
instead of~\eqref{eq:lr}. For each value of $a$ and each number of restarts, we report the final training and validation losses averaged over $5$ independent runs, together with one standard deviation. In all experiments presented here, we set $\hat f^*=0$.

\begin{figure}[ht]
    \centering
    \begin{tabular}{@{}cc@{}}
        \begin{minipage}[t]{0.48\textwidth}
            \centering
            \input{images/nanoGPT/nanoGPT_restarted_proximal_final_loss_vs_budget_per_tk.pgf}
        \end{minipage}
        &
        \begin{minipage}[t]{0.48\textwidth}
            \centering
            \input{images/nanoGPT/resnet_restarted_proximal_final_loss_vs_budget_per_tk.pgf}
        \end{minipage}
    \end{tabular}
    \caption{Final training and validation loss as a function of the number of restarts. The parameter $a$ is a multiplicative step size factor, which acts as an empirical estimate of the unknown theoretical scaling. {\bf Left:} nanoGPT trained on the works of Shakespeare. {\bf Right:} ResNet18 trained on CIFAR-10.}
    \label{fig:nn-restart-sensitivity}
  \end{figure}

\cref{fig:nn-restart-sensitivity} shows the final loss as a function of the number of restarts. In both experiments, the restart frequency has a visible effect on performance, and this effect strongly depends on the step size factor $a$. This indicates that the number of restarts and the step size scale should not be tuned independently.

For ResNet18, the clearest U-shaped behavior is observed for the restarted runs with $a=0.3$ and $a=1$. In these cases, increasing the number of restarts initially improves the final loss, but restarting too often eventually degrades performance. The same qualitative trade-off can be observed for nanoGPT trained on Shakespeare. Although the precise shape of the curves differs from that of the ResNet18 experiment, the plots again show that restart performance is sensitive to the choice of $a$ and to the number of restarts. Moderate restart frequencies can improve the final loss, whereas overly frequent restarts may increase both the loss and the variability across runs.

\bibliographystyle{agsm}
\bibliography{MainPerso,local}

@inproceedings{Mai20,
	author = {Mai, Vien and Johansson, Mikael},
	booktitle = {International conference on machine learning},
	organization = {PMLR},
	pages = {6630--6639},
	title = {Convergence of a stochastic gradient method with momentum for non-smooth non-convex optimization},
	year = {2020}}

@article{Masi26,
	author = {Masiha, Saeed and Kiyavash, Negar and Thiran, Patrick},
	journal = {arXiv preprint arXiv:2606.28278},
	title = {Sharp First-Order Lower Bounds under Sublevel alpha-Polyak-Lojasiewicz Conditions},
	year = {2026}}

@article{Masi24,
	author = {Masiha, Saeed and Salehkaleybar, Saber and He, Niao and Kiyavash, Negar and Thiran, Patrick},
	journal = {arXiv preprint arXiv:2408.01839},
	title = {Optimal Local Convergence Rates of Stochastic First-Order Methods under Local alpha-PL},
	year = {2024}}

@article{Davi18a,
	author = {Davis, Damek and Drusvyatskiy, Dmitriy and MacPhee, Kellie J and Paquette, Courtney},
	journal = {Journal of Optimization Theory and Applications},
	number = {3},
	pages = {962--982},
	publisher = {Springer},
	title = {Subgradient methods for sharp weakly convex functions},
	volume = {179},
	year = {2018}}

@article{Davi20a,
	author = {Davis, Damek and Drusvyatskiy, Dmitriy and Paquette, Courtney},
	journal = {IMA Journal of Numerical Analysis},
	number = {4},
	pages = {2652--2695},
	publisher = {Oxford University Press},
	title = {The nonsmooth landscape of phase retrieval},
	volume = {40},
	year = {2020}}

@article{Li23,
	author = {Li, Huan and Lin, Zhouchen},
	journal = {Journal of Machine Learning Research},
	number = {157},
	pages = {1--37},
	title = {Restarted Nonconvex Accelerated Gradient Descent: No More Polylogarithmic Factor in the in the O (epsilon\^{}(-7/4)) Complexity},
	volume = {24},
	year = {2023}}

@article{Losh16,
	author = {Loshchilov, Ilya and Hutter, Frank},
	journal = {arXiv preprint arXiv:1608.03983},
	title = {Sgdr: Stochastic gradient descent with warm restarts},
	year = {2016}}

@article{Kim18,
	author = {Kim, Donghwan and Fessler, Jeffrey A},
	journal = {Journal of Optimization Theory and Applications},
	number = {1},
	pages = {240--263},
	publisher = {Springer},
	title = {Adaptive restart of the optimized gradient method for convex optimization},
	volume = {178},
	year = {2018}}

@article{Wang22,
	author = {Wang, Bao and Nguyen, Tan and Sun, Tao and Bertozzi, Andrea L and Baraniuk, Richard G and Osher, Stanley J},
	journal = {SIAM Journal on Imaging Sciences},
	number = {2},
	pages = {738--761},
	publisher = {SIAM},
	title = {Scheduled restart momentum for accelerated stochastic gradient descent},
	volume = {15},
	year = {2022}}

@article{Zhou20,
	author = {Zhou, Yi and Wang, Zhe and Ji, Kaiyi and Liang, Yingbin and Tarokh, Vahid},
	journal = {arXiv preprint arXiv:2002.11582},
	title = {Proximal gradient algorithm with momentum and flexible parameter restart for nonconvex optimization},
	year = {2020}}

@article{Yang18,
	author = {Yang, Tianbao and Lin, Qihang},
	journal = {Journal of Machine Learning Research},
	number = {6},
	pages = {1--33},
	title = {Rsg: Beating subgradient method without smoothness and strong convexity},
	volume = {19},
	year = {2018}}

@article{Yu22,
	author = {Yu, Peiran and Li, Guoyin and Pong, Ting Kei},
	journal = {Foundations of Computational Mathematics},
	number = {4},
	pages = {1171--1217},
	publisher = {Springer},
	title = {Kurdyka--{\L}ojasiewicz exponent via inf-projection},
	volume = {22},
	year = {2022}}

@article{Bolt17a,
	author = {Bolte, J{\'e}r{\^o}me and Nguyen, Trong Phong and Peypouquet, Juan and Suter, Bruce W},
	journal = {Mathematical Programming},
	number = {2},
	pages = {471--507},
	publisher = {Springer},
	title = {From error bounds to the complexity of first-order descent methods for convex functions},
	volume = {165},
	year = {2017}}

@article{Atto13,
	author = {Attouch, Hedy and Bolte, J{\'e}r{\^o}me and Svaiter, Benar Fux},
	journal = {Mathematical programming},
	number = {1},
	pages = {91--129},
	publisher = {Springer},
	title = {Convergence of descent methods for semi-algebraic and tame problems: proximal algorithms, forward--backward splitting, and regularized Gauss--Seidel methods},
	volume = {137},
	year = {2013}}

@article{Fatk22,
	author = {Fatkhullin, Ilyas and Etesami, Jalal and He, Niao and Kiyavash, Negar},
	journal = {Advances in Neural Information Processing Systems},
	pages = {15836--15848},
	title = {Sharp analysis of stochastic optimization under global Kurdyka-Lojasiewicz inequality},
	volume = {35},
	year = {2022}}

@inproceedings{Font21,
	author = {Fontaine, Xavier and De Bortoli, Valentin and Durmus, Alain},
	booktitle = {Conference on Learning Theory},
	organization = {PMLR},
	pages = {1965--2058},
	title = {Convergence rates and approximation results for SGD and its continuous-time counterpart},
	year = {2021}}

@article{Hohe20,
	author = {Hoheisel, Tim and Laborde, Maxime and Oberman, Adam},
	doi = {10.3934/jdg.2020005},
	journal = {Journal of Dynamics and Games},
	number = {1},
	pages = {79--96},
	title = {A regularization interpretation of the proximal point method for weakly convex functions},
	volume = {7},
	year = {2020}}

@article{Drus19,
	author = {Drusvyatskiy, Dmitriy and Paquette, Courtney},
	journal = {Mathematical Programming},
	number = {1},
	pages = {503--558},
	publisher = {Springer},
	title = {Efficiency of minimizing compositions of convex functions and smooth maps},
	volume = {178},
	year = {2019}}

@article{Drus17,
	author = {Drusvyatskiy, Dmitriy},
	journal = {arXiv preprint arXiv:1712.06038},
	title = {The proximal point method revisited},
	year = {2017}}

@article{Davi18,
	author = {Davis, Damek and Drusvyatskiy, Dmitriy},
	journal = {arXiv preprint arXiv:1802.02988},
	title = {Stochastic subgradient method converges at the rate $ O (k^{-1/4}) $ on weakly convex functions},
	year = {2018}}

@article{Davi19,
	author = {Davis, Damek and Drusvyatskiy, Dmitriy and Charisopoulos, Vasileios},
	journal = {arXiv preprint arXiv:1907.09547},
	title = {Stochastic algorithms with geometric step decay converge linearly on sharp functions},
	year = {2019}}

@article{Vial83,
	author = {Vial, Jean-Philippe},
	journal = {Mathematics of Operations Research},
	number = {2},
	pages = {231--259},
	publisher = {INFORMS},
	title = {Strong and weak convexity of sets and functions},
	volume = {8},
	year = {1983}}

@article{Li18,
	author = {Li, Guoyin and Pong, Ting Kei},
	journal = {Foundations of computational mathematics},
	number = {5},
	pages = {1199--1232},
	publisher = {Springer},
	title = {Calculus of the exponent of Kurdyka--{\L}ojasiewicz inequality and its applications to linear convergence of first-order methods},
	volume = {18},
	year = {2018}}

@inproceedings{Roul17,
	author = {Roulet, Vincent and d'Aspremont, Alexandre},
	booktitle = {Advances in Neural Information Processing Systems},
	pages = {1119--1129},
	title = {Sharpness, restart and acceleration},
	year = {2017}}

@article{ODon15,
	author = {O'Donoghue, Brendan and Candes, Emmanuel},
	journal = {Foundations of computational mathematics},
	number = {3},
	pages = {715--732},
	publisher = {Springer},
	title = {Adaptive restart for accelerated gradient schemes},
	volume = {15},
	year = {2015}}

@inproceedings{Gise14,
	author = {Giselsson, Pontus and Boyd, Stephen},
	booktitle = {53rd IEEE Conference on Decision and Control},
	organization = {IEEE},
	pages = {5058--5063},
	title = {Monotonicity and restart in fast gradient methods},
	year = {2014}}

@article{Bolt07,
	author = {Bolte, J{\'e}r{\^o}me and Daniilidis, Aris and Lewis, Adrian},
	journal = {SIAM Journal on Optimization},
	number = {4},
	pages = {1205--1223},
	publisher = {SIAM},
	title = {The Lojasiewicz inequality for nonsmooth subanalytic functions with applications to subgradient dynamical systems},
	volume = {17},
	year = {2007}}

@inproceedings{Kurd98,
	author = {Kurdyka, Krzysztof},
	booktitle = {Annales de l'institut Fourier},
	pages = {769--783},
	title = {On gradients of functions definable in o-minimal structures},
	volume = {48},
	year = {1998}}

@article{Loja63,
	author = {Lojasiewicz, Stanislaw},
	journal = {Les {\'e}quations aux d{\'e}riv{\'e}es partielles},
	pages = {87--89},
	title = {Une propri{\'e}t{\'e} topologique des sous-ensembles analytiques r{\'e}els},
	year = {1963}}

@article{Nemi85,
	author = {Nemirovskii, AS and Nesterov, Yu E},
	journal = {USSR Computational Mathematics and Mathematical Physics},
	number = {2},
	pages = {21--30},
	publisher = {Elsevier},
	title = {Optimal methods of smooth convex minimization},
	volume = {25},
	year = {1985}}
\end{document}